\documentclass[11pt]{amsart}
\usepackage{a4wide}

\author{Manuel Bodirsky}
    \address{Laboratoire d'Informatique  (LIX), CNRS UMR 7161\\
    \'{E}cole Polytechnique \\91128 Palaiseau\\
    France}
    \email{bodirsky@lix.polytechnique.fr}
    \urladdr{http://www.lix.polytechnique.fr/~bodirsky/}
\author{Michael Pinsker}
    \address{\'{E}quipe de Logique Math\'{e}matique\\ Universit\'{e} Diderot-Paris 7\\
	UFR de Math\'{e}matiques\\ 
	75205 Paris Cedex 13, France}
    \email{marula@gmx.at}
    \urladdr{http://dmg.tuwien.ac.at/pinsker/}
    \thanks{The research leading to these results has received funding from the European Research Council under the European Community's Seventh Framework Programme (FP7/2007-2013 Grant Agreement no. 257039). The second author is grateful for support through Erwin Schr\"{o}dinger Fellowship J2742-N18 of the Austrian Science Fund and through an APART-fellowship of the Austrian Academy of Sciences.}

\title[Minimal Functions on the Random Graph]{Minimal Functions on the Random Graph}
\subjclass[2000]{Primary 03C10; secondary 05C80; 08A35; 05C55;
03C40}

\date{\today}
\keywords{random graph, $\omega$-categoricity, minimal function, Ramsey theory, automorphism, endomorphism,
polymorphism, local clone,
model-completeness, first-order definition, existential positive definition, primitive positive definition}

\usepackage{cite}
\usepackage{amsmath, amsxtra, amsfonts, amssymb}
\usepackage{mathrsfs,amsthm,url}

\theoremstyle{plain}
    \newtheorem{thm}{Theorem}
    \newtheorem{theorem}[thm]{Theorem}
    \newtheorem{lem}[thm]{Lemma}
    \newtheorem{prop}[thm]{Proposition}
    \newtheorem{proposition}[thm]{Proposition}
    \newtheorem{cor}[thm]{Corollary}

\theoremstyle{definition}
    \newtheorem{defn}[thm]{Definition}
    \newtheorem{definition}[thm]{Definition}
    
\theoremstyle{remark}

\DeclareMathOperator{\sw}{sw}
\DeclareMathOperator{\flip}{--}
\newcommand{\ignore}[1]{}
\newcommand{\nin}{\notin}
\providecommand{\Aut}{\mathop{\rm Aut}\nolimits}
\providecommand{\End}{\mathop{\rm End}\nolimits}
\providecommand{\Emb}{\mathop{\rm Emb}\nolimits}
\DeclareMathOperator{\typ}{tp_{qf}} 
\DeclareMathOperator{\id}{id}
\newcommand{\mult}{\times}

\newcommand{\Q}{{\mathcal Q}}
\newcommand{\U}{{\mathcal U}}
\newcommand{\C}{{\mathcal C}}
\newcommand{\R}{{\mathcal R}}
\newcommand{\F}{{\mathcal F}}

\newcommand{\A}{{\mathcal A}}
\newcommand{\B}{{\mathcal B}}
\renewcommand{\P}{{\mathcal P}}

\newcommand{\G}{{\mathcal G}}
\newcommand{\M}{{\mathcal M}}

\newcommand{\T}{{\mathcal T}}

\newcommand{\K}{{\mathcal K}}
\renewcommand{\H}{{\mathcal H}}

\renewcommand{\S}{{\mathcal S}}
\newcommand{\To}{\rightarrow}
\newcommand{\cal}[1]{\ensuremath{\mathcal{#1}}}

\begin{document}

\begin{abstract}
We show that there is a system of 14 non-trivial finitary functions on the random graph with the following properties: Any non-trivial function on the random graph generates one of the functions of this system by means of composition with automorphisms and by topological closure, and the system is minimal in the sense that no subset of the system has the same property. The theorem is obtained by proving a Ramsey-type theorem for colorings of tuples in finite powers of the random graph, and by applying this to find regular patterns in the behavior of any function on the random graph. As model-theoretic corollaries of our methods we rederive a theorem of Simon Thomas classifying the first-order closed reducts of the random graph, and prove some refinements of this theorem; also, we obtain a classification of the maximal reducts closed under primitive positive definitions, and prove that all reducts of the random graph are model-complete.
\end{abstract}
\maketitle

\section{Introduction}\label{sect:prelims}

\subsection{The random graph}
The \emph{random graph} (also called the \emph{Rado graph}) is the countably infinite graph $G=(V;E)$
defined uniquely up to isomorphism by the \emph{extension property}: for all
finite disjoint subsets $U,U'$ of the countably infinite vertex set
$V$ there exists a vertex $v \in V \setminus (U \cup U')$ such that
$v$ is in $G$ adjacent to all vertices in $U$ and to no vertex in
$U'$. Alternatively, $G$ is the unique countable graph which is \emph{universal} in the sense that it contains all finite graphs as induced subgraphs, and  \emph{homogeneous} in the sense that any isomorphism between finite induced subgraphs of $G$ extends to an automorphism of $G$. 
For the many remarkable properties of
$G$ and its automorphism group $\Aut(G)$, and various connections to
many branches of mathematics, see
e.g.~\cite{RandomCameron,RandomRevisitedCameron}.

\subsection{Minimal functions}

We say that a finitary operation $f \colon V^k \rightarrow V$
\emph{generates} an operation $g \colon  V^l \rightarrow V$
iff $g$ is contained in the topological closure of the set of term functions that can be built from $f$ and the automorphisms of $G$, where the topology on functions is just the pointwise convergence topology -- we refer to Section~\ref{sect:results} for a more technical definition.

By this relation of function generation, the functions on the random graph are quasi-ordered with respect to their ``generating strength''. The weakest functions in this order are the \emph{trivial functions}, which we define to be those functions that are generated by the identity $\id \colon V\To V$; these trivial functions are the self-embeddings of $G$ with possibly additional dummy variables. On the next level are the \emph{minimal functions}: An operation $f$ is called \emph{minimal} 
iff it is non-trivial, and all non-trivial functions $g$ it generates have at least the arity of $f$ and generate $f$. Simon Thomas proved in \cite{RandomReducts} that there are exactly two minimal unary bijective
operations on the random graph that do not generate each other. In this paper
we generalize this to arbitrary finitary operations, and show that there are
exactly 14 minimal operations on $G$ that do not generate each other:
these are a constant operation, an operation that maps $G$ injectively to a complete subgraph of $G$, an operation that maps $G$ injectively to an independent subset of $G$, the two operations considered by Thomas, and nine binary
injective operations.

\subsection{Ramsey theory}

In our proof, we apply in a systematic way structural Ramsey theory.
Any function $f \colon V\To V$ induces a coloring of the edges of the random graph by three colors: each edge might either be sent to an edge, to a non-edge, or be collapsed to a single vertex. Similarly, $f$ induces a coloring of the non-edges. If $f$ is not unary, but a function from a power $V^k$ to $V$, then it induces colorings of pairs of elements of $V^k$ of a fixed type, and so on. We will use a theorem of Ne\v{s}et\v{r}il and R\"{o}dl from ~\cite{NesetrilRoedlPartite,NesetrilRoedlOrderedStructures} (and independently by~\cite{AbramsonHarrington}) which states
 that finite ordered vertex-colored
graphs form a Ramsey class in order to prove a Ramsey-type theorem which in turn allows us to find regular patterns in these colorings, for any function $f$. This makes it feasible to understand the generating process of functions; in particular, all minimal functions turn out to have \emph{canonical} behavior in the sense that seen from the right perspective, the colorings they induce are all constant.

\subsection{Universal algebra: Clones}
In universal algebra, a \emph{clone} on a set $D$ is a subset of the set
$\mathcal O$ of all finitary operations on $D$ which is closed under composition of operations and which contains the projections (see e.g.~\cite{Szendrei,KaluzninPoeschel}). 
The set of all clones over $D$ forms a complete lattice with respect to set-theoretical inclusion; see~\cite{GoldsternPinsker} for a survey of results on this lattice for infinite $D$. In many applications, however, e.g., in theoretical computer science, 
one does not study the lattice of all clones, but the smaller lattice of those clones which are closed in the
the pointwise convergence topology on $\mathcal O$. This topology is given by the countable basis of sets of the form
$${\cal O}^s_A:=\{f\in {\cal O}\; | \; f|_A = s\},$$
where $A\subseteq D^n$ is finite and $s \colon A\To D$ is a function. Clones which are closed subsets of
$\mathcal O$ in this topology are called
\emph{locally closed}, or just \emph{local}. The importance of such clones is reminiscent of that of closed permutation groups (rather than arbitrary permutation groups) for some applications (see e.g.~\cite{Oligo}).

The lattice of local clones has been studied in~\cite{Pin-morelocal}, and it turns
out to be quite complicated. However, one is often interested in specific parts of this lattice, in particular the interval of all local clones that contain $\Aut(\Gamma)$, for a structure $\Gamma$. When $\Gamma$ is \emph{$\omega$-categorical}, i.e., every countable model of the first-order theory of $\Gamma$ is isomorphic to $\Gamma$, then it turns out that  techniques similar to those for clones over
finite domains can be applied for the study of this interval, in particular when classifying its atoms ~\cite{OligoClone,BodChenPinsker}.

If we consider the lattice of local clones
containing $\Aut(G)$, 
then it is easy to see that a clone $\C$ is an atom in this lattice iff there exists a minimal operation $f$ on $G$ such that $\C$ is the smallest local clone containing the set  $\{f\} \cup \text{Aut}(G)$; we then say that $\C$ is the local clone generated by $\{f\}$ over $\text{Aut}(G)$. Hence, in this paper we determine all atoms of the lattice of local clones containing $\Aut(G)$.

\subsection{Groups and monoids}

Similarly to clones, the (topologically) closed permutation groups containing $\Aut(G)$ form a complete lattice, with the meet of a set of groups being their intersection; so do the closed transformation monoids containing $\Aut(G)$. By determining the minimal functions on the random graph we find the atoms not only of the lattice of local clones containing $\Aut(G)$, but also of the corresponding group and monoid lattices. In the group case, it turns out that if one continues ``climbing up" in the lattice, i.e., if after the atoms of the lattice one determines the next level and so on, one finds the whole lattice as the lattice has only five elements. This was shown already by Thomas in \cite{RandomReducts}, and we will rederive this result. Our methods also allow to follow the same strategy for the other two lattices, but the iteration does not terminate as these lattices have infinite height.

\subsection{Model theory: Reducts of the random graph}\label{subsect:intro:model}

Results
about operations on the random graph $G$ yield model-theoretic results about \emph{reducts} of $G$, i.e.,
about relational structures with the same domain as $G$ whose relations have a first-order definition in $G$; this is particularly true because $G$ is $\omega$-categorical. In fact, if we consider two reducts equivalent iff they first-order define one another, then the lattice of all reducts, factored by this equivalence, is antiisomorphic to the lattice of closed permutation groups that contain $\Aut(G)$. Similarly, the finer lattice of reducts up to \emph{existential positive interdefinability} corresponds to the lattice of closed transformation monoids that contain $\Aut(G)$. Finally, the lattice of reducts factored by the even finer equivalence of \emph{primitive positive interdefinability} (a first-order formula is \emph{primitive positive} iff it contains no negations, disjunctions, and universal quantifications), corresponds to the lattice of closed clones that contain $\Aut(G)$. Using the latter connection, we obtain a list of the dual atoms in the lattice of reducts up to primitive positive interdefinability; 
there are 14 such dual atoms, each corresponding to one of the minimal operations mentioned above.

As another application of the techniques in this paper we
rederive the full result of Thomas from~\cite{RandomReducts},
which is in fact a classification of
the reducts of $G$ up to \emph{first-order interdefinability}. We show that the result can be strengthened to obtain a classification of
those structures up to \emph{existential interdefinability}. Finally, we show that all reducts $\Gamma$ of $G$ have a model-complete theory. 

\subsection{Computational complexity: Constraint satisfaction}
Many computational problems in theoretical computer science can be elegantly
formalized in the following way. Fix a structure $\Gamma$
with finite relational signature. Then the \emph{constraint satisfaction problem for $\Gamma$} (CSP$(\Gamma)$) is the problem of deciding whether a given
primitive positive sentence is true in $\Gamma$. The computational complexity of CSP$(\Gamma)$
has been determined for all two-element structures in~\cite{Schaefer},
for all three-element structures in~\cite{Bulatov}, and for all
structures with a first-order definition in $(\mathbb Q;<)$ in~\cite{tcsps-journal}. The results of the present paper provide the necessary
mathematical techniques for a complexity classification for CSP$(\Gamma)$
when $\Gamma$ has a first-order definition in the random graph~\cite{BodPin-Schaefer}; such constraint satisfaction problems constitute a generalization of Boolean constraint satisfaction problems to the language of graphs.
We have to refer to the introduction of~\cite{BodChenPinsker}, the survey paper~\cite{BP-reductsRamsey}, or the papers~\cite{BodPinTsa} and ~\cite{Topo-Birk} for a more detailed description of the general connection of reducts of $\omega$-categorical structures with the CSP.

\subsection{Structure of this paper}

This introduction will be followed by Section~\ref{sect:results} in which we present our results on functions on the random graph in full detail. We then discuss these results from a model-theoretic perspective and draw some corollaries in Section~\ref{sect:model}; this section can be skipped by anyone not interested in the matter.  
The proof of the results in Section~\ref{sect:results} starts with Section~\ref{ssect:ramsey}, where we recall some Ramsey-type theorems and extend them for our purposes. We then apply these theorems to mappings from $V$ to $V$ in order to get hold of such mappings in Section~\ref{sect:structureInMappings}. This allows us to determine the minimal unary functions in Section~\ref{sect:minimalUnary}. Turning to functions of higher arity in Section~\ref{sect:reducingArity} , we show that minimal higher arity functions are always binary injections. In order to understand binary minimal functions, we develop further Ramsey-theoretic tools in Section~\ref{sect:productRamsey}. Finally, in Section~\ref{sect:minimalBinary}, we determine the minimal binary injections, completing our proof.

\section{Results}\label{sect:results}

\subsection{The minimal functions result} 
When $f \colon V^n \rightarrow V$ and $g_1, \dots,g_n\colon V^m \rightarrow V$
are operations, then the \emph{composition} of $f$ with $g_1,\dots,g_n$
is the operation defined by
$$(x_1,\dots,x_m) \mapsto f(g_1(x_1,\dots,x_m),\dots,g_n(x_1,\dots,x_m)) \; .$$
An operation $f \colon V^n \rightarrow V$ is called a \emph{projection} iff there exists $1\leq i \leq n$ such
that $f(x_1,\dots,x_n)=x_i$ for all $x_1,\dots,x_n \in V$.

The following definition of the notion \emph{generates} is equivalent to the one in the introduction.

\begin{defn}
Let  $f \colon V^k \rightarrow V$ and $g \colon  V^l \rightarrow V$. We say that $f$
\emph{generates} $g$ 
iff for every finite subset $S$ of $V^l$ the restriction of $g$ to $S$ equals the
restriction to $S$ of an $l$-ary operation that can be obtained from $f$, automorphisms of $G$,
and the projections by a sequence of compositions of operations.
\end{defn}

We also give the definition of a \emph{minimal function} in full detail. 

\begin{defn}Let $f \colon V^k \rightarrow V$ and $g \colon  V^l \rightarrow V$.
\begin{itemize}
\item $f$ and $g$
are \emph{equivalent} iff $f$ generates $g$ and $g$ generates $f$.
\item $g$  is \emph{trivial} iff it is equivalent to the identity function on $V$.
\item $f$ is \emph{minimal} iff it is not trivial, and all non-trivial functions $g$ generated by $f$ have arity at least $k$ and are equivalent to $f$.
\end{itemize}
\end{defn}

We now define a small number of special operations on $G$.
The random graph contains all countable graphs as induced subgraphs, and in particular,
it contains an infinite complete subgraph, denoted by
$K_\omega$.
It follows from the homogeneity of $G$
that all injective operations
from $V$ to $V$ whose image induces $K_\omega$ in $G$
generate each other. Let $e_E$ be one such injective operation. 

We define $N := \{(x,y) \in V^2 \; | \; (x,y) \notin E\,\wedge\, x \neq y\}$. Pairs $\{x,y\}$ with $(x,y) \in N$ are referred to as \emph{non-edges}. 
$G$ contains an
infinite independent set, denoted by $I_\omega$. Let $e_N$ be an
injective operation from $V$ to $V$ whose image induces
$I_\omega$ in $G$.

It is clear that the complement graph of $G$, i.e., the graph on $V$ obtained by flipping all edges and non-edges of $G$, is isomorphic to $G$.
Again, note that by homogeneity of $G$ all isomorphisms between $G$ and
its complement generate each other.
Let $\flip$ be one such isomorphism. In formulas, we will 
write $-x$ for $\flip(x)$. 

For any finite non-empty subset $S$ of $V$, if we flip edges and non-edges
between $S$ and $V \setminus S$ in $G$, then the resulting graph is
isomorphic to $G$ (it is straightforward to verify the extension
property). All such isomorphisms generate each other.
For each non-empty
finite $S$, we let $i_S$ be such an isomorphism.
We also write $\sw$ for $i_{\{0\}}$, where $0 \in V$
is a fixed element for the rest of the paper, and refer to this
operation as the \emph{switch}.

To take a break from the definitions, we state the first part of our main theorem, which characterizes the minimal unary functions on $G$.

\begin{theorem}\label{thm:mainUnary}
Any minimal unary function on $G$ is equivalent to exactly one of the following operations:
\begin{enumerate}
\item a constant operation;
\item $e_N$;
\item $e_E$;
\item $\flip$;
\item $\sw$.
\end{enumerate}
\end{theorem}

We now turn to minimal functions of higher arity. 

\begin{defn}
Let $f \colon V^2 \rightarrow V$ be a binary injective
operation. 

The \emph{dual} $f^*$
of $f$ is defined by $f^*(x,y)=-f(-x,-y)$.

We say that $f \colon V^2 \rightarrow V$ is
\begin{itemize}
\item \emph{of type $p_1$} iff for all $x_1,x_2,y_1,y_2 \in V$ with $x_1 \neq x_2$
and $y_1 \neq y_2$ we have\\
$(f(x_1,y_1),f(x_2,y_2)) \in E$ if and only if
$(x_1,x_2) \in E$;
\item \emph{of type $\max$} iff for all $x_1,x_2,y_1,y_2 \in V$ with $x_1 \neq x_2$ and $y_1 \neq y_2$ we have\\
$(f(x_1,y_1),f(x_2,y_2)) \in E$ if and only if
$(x_1,x_2) \in E$ or $(y_1,y_2) \in E$;
\item \emph{balanced in the first argument} iff for all $x_1,x_2,y \in V$
with $x_1 \neq x_2$ we have\\ $(f(x_1,y),f(x_2,y)) \in E$ if and only if $(x_1,x_2) \in E$;
\item \emph{balanced in the second argument}
iff $(x,y) \mapsto f(y,x)$ is balanced in the first argument;
\item \emph{$E$-dominated in the first argument}
iff for all $x_1,x_2,y \in V$ with $x_1 \neq x_2$ we have that $(f(x_1,y),f(x_2,y)) \in E$;
\item \emph{$E$-dominated in the second argument} iff
$(x,y) \mapsto f(y,x)$ is $E$-dominated in the first argument.
\end{itemize}
\end{defn}

We can now state our main result.
\begin{theorem}\label{thm:main}
Any minimal function on $G$ is equivalent to one of the unary operations in Theorem~\ref{thm:mainUnary}, or to exactly one of the following operations:
\begin{enumerate}
\item [(6)] a binary injection of type $p_1$ that is balanced in both arguments;
\item [(7)] a binary injection of type $\max$ that is balanced in both arguments;
\item [(8)] a binary injection of type $\max$ that is $E$-dominated in both arguments;
\item [(9)] a binary injection of type $p_1$ that is $E$-dominated in both arguments;
\item [(10)] a binary injection of type $p_1$ that is balanced in the first
and $E$-dominated in the second argument;
\end{enumerate}
or to one of the duals of the last four operations (the dual of the operation in (6) is equivalent to the operation itself).
\end{theorem}

\subsection{Results on groups and monoids}
The technique to show this result can be applied several times to
unary bijective operations in order to rederive a result by Simon Thomas (Theorem~\ref{thm:reducts} below).
A permutation group $\G$ acting on a set $D$ is called
 \emph{(locally) closed} iff it is closed in the space of all permutations on $D$ equipped with the pointwise convergence topology; equivalently, $\G$ contains
all permutations which can be interpolated by elements of $\G$ on
arbitrary finite subsets of $D$. We call the smallest closed group
containing a set of permutations $\cal F$ on $V$ as well as $\Aut(G)$ the \emph{group generated by $\cal F$}.

\begin{theorem}[from~\cite{RandomReducts}]\label{thm:reducts}
The closed permutation permutation groups containing $\Aut(G)$ are precisely the following.
\begin{enumerate}
\item $\Aut(G)$;
\item the group generated by $\{\flip\}$;
\item the group generated by $\{\sw\}$;
\item the group generated by $\{\flip,\sw\}$;
\item the group of all permutations on $V$.
\end{enumerate}
\end{theorem}

The arguments given in~\cite{RandomReducts} use a Ramsey-theoretic
result by Ne\v{s}et\v{r}il~\cite{NesRod:Partitions}, namely that the
class of all finite graphs excluding finite cliques of a fixed size
forms a \emph{Ramsey class} (in the sense of~\cite{NesetrilSurvey}).
We also use a Ramsey-theoretic result, shown by R\"odl and
Ne\v{s}et\v{r}il~\cite{NesetrilRoedlPartite,NesetrilRoedlOrderedStructures}
(and independently by~\cite{AbramsonHarrington}), which is
different: we need the fact that finite ordered vertex-colored
graphs form a Ramsey class.

Similarly to groups and clones, a monoid $\M$ of operations from a set $D$ to $D$ is called
\emph{(locally) closed} iff it is closed in the space $D^D$ equipped with the pointwise convergence topology.
Our proof moreover shows the following statement about closed transformation
monoids that contain Aut$(G)$;
this statement
also follows from another combinatorial proof
of Simon Thomas given in~\cite{Thomas96} (where he uses the notion of \emph{pseudo-reducts} instead of
a formulation in terms of closed monoids). 

\begin{theorem}\label{thm:endos}
For any closed monoid $\M$ containing $\Aut(G)$, one of the
following cases applies.
\begin{enumerate}
\item $\M$ contains a constant operation.
\item $\M$ contains $e_E$.
\item $\M$ contains $e_N$.
\item The permutations in $\M$ form a group which is a dense subset of $\M$ in the space $V^V$.
\end{enumerate}
\end{theorem}

\section{Model-theoretic corollaries}\label{sect:model}

We now discuss a model-theoretic interpretation of these results 
as well as further model-theoretic consequences; this section can be skipped without affecting readability of the rest of the paper.

Since $G$ is homogeneous in a finite language it is $\omega$-categorical
(Corollary~6.4.2 of \cite{Hodges}). The reducts of a countable $\omega$-categorical
structure $\Gamma$ 
are $\omega$-categorical (see e.g.
\cite{Hodges}). In particular, this is true for all reducts of $G$.

We say that two structures $\Gamma$ and $\Delta$ on the same domain are \emph{first-order interdefinable} when $\Gamma$ is first-order definable in $\Delta$ and vice versa. 
Let $f \colon D^n \rightarrow D$ be an operation
and let $R \subseteq D^m$ be a relation. For tuples $r_1,\ldots,r_n\in D^m$ we write $f(r_1,\ldots,r_n)$ for the $m$-tuple that is obtained by applying $f$ to $r_1,\ldots, r_n$ componentwise, i.e., for the $m$-tuple whose $i$-th component is $f(r_1^i,\ldots,r_n^i)$, where $r_j^i$ denotes the $i$-th component of $r_j$, for $1\leq j\leq m$ and $1\leq i\leq m$. We say
that $f$ \emph{preserves} $R$ iff $f(r_1,\ldots,r_n) \in R$
whenever $r_1,\ldots ,r_n\in R$, and that $f$ \emph{violates} $R$ otherwise.
The theorem of Engeler,
Ryll-Nardzewski, and Svenonius (see
e.g.~\cite[Theorem~6.3.1]{Hodges}) implies that a relation $R$ is
first-order definable in a countable $\omega$-categorical structure $\Delta$
 if and only if $R$ is preserved by all automorphisms of $\Delta$. As a consequence, the reducts of a countable
$\omega$-categorical structure $\Delta$ are, up to first-order
interdefinability, in one-to-one correspondence with the locally
closed permutation groups containing $\Aut(\Delta)$. To illustrate
this, we restate Theorem~\ref{thm:reducts} by means of this
connection.

On the random graph, let $R^{(k)}$ be the $k$-ary relation that
holds on $x_1,\dots,x_k \in V$ iff $x_1,\dots,x_k$ are pairwise
distinct, and the number of edges between these $k$ vertices is odd.
Note that $R^{(4)}$ is preserved by $\flip$, $R^{(3)}$ is preserved by
$\sw$, and that $R^{(5)}$ is preserved by $\flip$ and by $\sw$, but not
by all permutations of $V$.

\begin{theorem}[Simon Thomas~\cite{RandomReducts}]\label{thm:reducts2}
Any reduct of $G$ is first-order interdefinable with precisely one of the following structures.
\begin{enumerate}
\item $G=(V;E)$;
\item $(V;R^{(4)})$;
\item $(V;R^{(3)})$;
\item $(V;R^{(5)})$;
\item $(V;=)$.
\end{enumerate}
\end{theorem}

For any reduct $\Gamma$ of $G$, a case of Theorem~\ref{thm:reducts2}
applies iff the case with the same number applies for $\Aut(\Gamma)$
in Theorem~\ref{thm:reducts}. We will not prove this relational
description in this paper; however, given Theorem~\ref{thm:reducts}
and the discussion above, verifying the equivalence is merely an
exercise.

In the same way as automorphisms of a countable $\omega$-categorical structure $\Delta$ can be used to characterize
first-order definability in $\Delta$, self-embeddings of $\Delta$ (that is, embeddings of $\Delta$ into itself) can be used to
characterize existential definability, endomorphisms of $\Delta$ can be used
to characterize existential positive definability, and polymorphisms of $\Delta$ (i.e., homomorphisms from a finite power $\Delta^n$ to $\Delta$, or simply finitary operations preserving all relations of $\Delta$) can be used to characterize primitive positive definability in
$\Delta$.

A first-order formula $\phi$ is called \emph{existential}
iff it is of the form $\exists x_1,\dots,x_k. \; \psi$, where $\psi$ is
quantifier-free. If $\psi$ is even of the form $\psi_1 \wedge \cdots \wedge \psi_m$ for atomic formulas $\psi_1,\dots,\psi_m$,
then $\phi$ is called \emph{primitive positive}. 
A formula is called \emph{existential positive}
iff it is a disjunction of primitive positive formulas. 
Call two structures $\Gamma$ and $\Delta$ \emph{primitive positive interdefinable} iff every relation in
$\Gamma$ has a definition by a primitive positive formula in $\Delta$ and vice versa; we have analogous definitions for existential positive and existential interdefinability.
To translate results about operations on $G$
into results about primitive positive definability in reducts of $G$, the following theorem is central.

\begin{theorem}[from~\cite{BodirskyNesetrilJLC}]\label{thm:nesetril}
Let $\Gamma$ be a countable $\omega$-categorical structure. Then a relation $R$ is primitive positive definable in $\Gamma$ if and only if
$R$ is preserved by all polymorphisms of $\Gamma$.
\end{theorem}

The operational generating
process can be linked to preservation of relations of structures~\cite{BodirskyNesetrilJLC,OligoClone,BodChenPinsker}.

\begin{proposition}\label{prop:galois}
Let $f \colon V^k \rightarrow V$ and $g \colon  V^l \rightarrow V$ be operations.
Then $f$ generates $g$ if and only if every relation with a first-order definition in $G$ that is preserved by $g$ is also preserved by $f$.
\end{proposition}

Recall from Subsection~\ref{subsect:intro:model} that the reducts of $G$, factored by the equivalence of primitive positive interdefinability, form a complete lattice in which is the order is given by primitive positive definability. 
Using Theorem~\ref{thm:nesetril} and Proposition~\ref{prop:galois},  we obtain the following
equivalent formulation of Theorem~\ref{thm:main}.
 The \emph{dual} of $\Gamma$ is the structure
that consists of the relations $-R := \{(-t_1,\dots,-t_k) \; | \; (t_1,\dots,t_k) \in R\}$ for all relations $R$ in $\Gamma$.

\begin{cor}\label{thm:mainForStructures}
Let $\Gamma$ be a member of a dual atom of the lattice of reducts of $G$ up to primitive positive interdefinability.
Then it is primitive positive interdefinable  with exactly one of the following 14 structures, namely the structures with all relations that are first-order definable in $G$ and preserved 
by \begin{enumerate}
\item a constant operation;
\item $e_N$;
\item $e_E$;
\item $\flip$;
\item $\sw$;
\item a binary operation of type $p_1$ that is balanced in both arguments;
\item a binary operation of type $\max$ that is balanced in both arguments;
\item a binary operation of type $\max$ that is $E$-dominated in both arguments;
\item a binary operation of type $p_1$ that is $E$-dominated in both arguments;
\item a binary operation of type $p_1$ that is balanced in the first
and $E$-dominated in the second argument;
\end{enumerate}
or to one of the duals of the last four structures.
\end{cor}

Existential positive and existential definability
in a countable $\omega$-categorical structure $\Gamma$ can be described in terms of the endomorphism monoid of $\Gamma$.
\begin{proposition}\label{thm:pres}
A relation $R$ has an existential positive (existential) definition
in a countable $\omega$-categorical structure $\Gamma$ if and only if $R$ is
preserved by the endomorphisms (self-embeddings) of $\Gamma$.
\end{proposition}
\begin{proof}
It is easy to verify that existential positive formulas are
preserved by endomorphisms, and existential formulas are preserved
by self-embeddings of $\Gamma$.

For the other direction, note that the endomorphisms and
self-embeddings of $\Gamma$ contain the automorphisms of $\Gamma$,
and hence the theorem of Ryll-Nardzewski shows that $R$ has a
first-order definition in $\Gamma$; let $\phi$ be a formula defining
$R$. Suppose for contradiction that $R$ is preserved by all
endomorphisms of $\Gamma$ but has no existential positive definition
in $\Gamma$. We use the homomorphism preservation theorem (see
\cite[Section 5.5, Exercise~2]{Hodges}),
 which states that a
first-order formula $\phi$ is equivalent to an existential positive
formula modulo a first-order theory $T$ if and only if $\phi$ is
preserved by all homomorphisms between models of $T$. Since by
assumption $\phi$ is not equivalent to an existential positive
formula in $\Gamma$, there are models $\Gamma_1$ and $\Gamma_2$ of
the first-order theory of $\Gamma$ and a homomorphism $h$ from
$\Gamma_1$ to $\Gamma_2$ that violates $\phi$. By the Theorem of
L\"owenheim-Skolem (see e.g.~\cite{Hodges}) the first-order theory
of the two-sorted structure $(\Gamma_1,\Gamma_2;h)$ has a countable
model $(\Gamma_1',\Gamma_2';h')$. Since both $\Gamma_1'$ and
$\Gamma_2'$ must be countably infinite, and because $\Gamma$ is
$\omega$-categorical, we have that $\Gamma_1'$ and $\Gamma_2'$ are
isomorphic to $\Gamma$, and $h'$ can be seen as an endomorphism of
$\Gamma$ that violates $\phi$; a contradiction.

The argument for existential definitions and self-embeddings is
similar, but instead of the homomorphism preservation theorem we use
the Theorem of {\L}os-Tarski which states that a first-order formula
$\phi$ is equivalent to an existential formula modulo a first-order
theory $T$ if and only if $\phi$ is preserved by all embeddings
between models of $T$ (see e.g.~\cite[Corollary 5.4.5]{Hodges}).
\end{proof}

Using Proposition~\ref{thm:pres}, we obtain an interesting and perhaps
surprising consequence of Theorem~\ref{thm:endos}. A theory $T$ is called
\emph{model-complete} iff every embedding between models of $T$ is
elementary, i.e., preserves all first-order formulas. It is
well-known that a theory $T$ is model-complete if and only if every
first-order formula is modulo $T$ equivalent to an existential
formula (see~\cite[Theorem 7.3.1]{Hodges}). A structure is said to
be model-complete iff its first-order theory is model-complete. From
the definition of model-completeness and $\omega$-categoricity it is
easy to see that a countable $\omega$-categorical structure $\Gamma$ is
model-complete iff all self-embeddings of $\Gamma$ preserve
all first-order formulas. We write $\Emb(\Gamma)$ for the monoid of all self-embeddings of $\Gamma$. 

\begin{lem}\label{lem:mc}
A countable $\omega$-categorical structure $\Gamma$ is model-complete if and
only if $\Aut(\Gamma)$ is dense in $\Emb(\Gamma)$. 
\end{lem}
\begin{proof}
First assume that all self-embeddings of $\Gamma$ are in the topological closure of $\Aut(\Gamma)$. Let $\phi$ be a first-order formula. By the equivalent
characterization of model-completeness mentioned above it suffices
to show that $\phi$ is equivalent to an existential formula. Since
$\phi$ is preserved by automorphisms of $\Gamma$, it is also preserved by self-embeddings of
$\Gamma$. Then Proposition~\ref{thm:pres} implies that $\phi$ is
equivalent to an existential formula.

Conversely, suppose that all
first-order formulas are equivalent to an existential formula in
$\Gamma$. Since existential formulas are preserved by
self-embeddings of $\Gamma$, also the first-order formulas are
preserved by self-embeddings of $\Gamma$.
Then the theorem of Engeler, Ryll-Nardzewski, and Svenonius shows that
every relation that is preserved by all automorphisms of $\Gamma$
is also preserved by the self-embeddings of $\Gamma$. Now if there were a self-embedding $e$ not in the closure of $\Aut(\Gamma)$, then
there would be a finite tuple $t$ in $\Gamma$ such that $e(t)\neq\alpha(t)$ for all $\alpha\in\Aut(\Gamma)$. Let $R:=\{\alpha(t) \; | \; \alpha\in\Aut(\Gamma)\}$ . Then $R$ is preserved by all automorphisms of $\Gamma$ but not by $e$, a contradiction.
\end{proof}

It follows from a result
in~\cite[Proposition 19]{tcsps-journal} (based on a proof of a result by
Cameron~\cite{Cameron5} from~\cite{JunkerZiegler}) that all reducts
of the linear order of the rationals $(\mathbb Q;<)$ are
model-complete. We now see that the same is true for the random
graph. Recall that the homogeneity of $G$ implies that it has \emph{quantifier-elimination}:
every first-order formula is in $G$ equivalent to a quantifier-free
first-order formula.

\begin{cor}\label{cor:mc}
All reducts of the random graph are model-complete.
\end{cor}
\begin{proof}
Let $\Gamma$ be a reduct. We apply Theorem~\ref{thm:endos} to $\Emb(\Gamma)$. If Case (4) of the
theorem holds, then we are done by Lemma~\ref{lem:mc}. Note that $\Emb(\Gamma)$ cannot contain a constant operation as all its operations are injective. So suppose
that $\Emb(\Gamma)$ contains $e_N$ (the argument for $e_E$ is analogous). Let $R$ be any relation of $\Gamma$, and $\phi_R$ be its defining quantifier-free formula. Let $\psi_R$ be the formula obtained by replacing all occurrences of $E$ by \emph{false}; so $\psi_R$ is a formula over the empty language. Then a tuple $a$ satisfies $\phi_R$ in $G$ iff $e_N(a)$ satisfies $\phi_R$ in $G$ (because $e_N$ is an embedding) iff $e_N(a)$ satisfies $\psi_R$ in $G$ (as there are no edges on $e_N(a)$) iff $e_N(a)$ satisfies $\psi_R$ in the substructure induced by $e_N[V]$ (since $\psi_R$ does not contain any quantifiers). Thus, $\Gamma$ is isomorphic to the structure on $e_N[V]$ which has the relations defined by the formulas $\psi_R$; hence, $\Gamma$ is isomorphic to a structure with a first-order definition over the empty language. This structure has, of course, all injections as self-embeddings, and all permutations as automorphisms, and hence is model-complete by Lemma~\ref{lem:mc}; thus, the same is true for $\Gamma$.
\end{proof}

Although $G$ has quantifier-elimination, the same is not true for its reducts. For example, any
two 2-element substructures of the structure
$$\Gamma = (V; \{(x,y,z) \; | \; (x,y) \in E \wedge (y,z) \in N\})$$
are isomorphic. But since there is a first-order definition of $G$
in $\Gamma$, an isomorphism between a 2-element substructure with an
edge and a 2-element substructure without an edge cannot be extended
to an automorphism of $\Gamma$. However, our results imply that a
structure $\Gamma$ with a first-order definition in the random graph
is homogeneous when $\Gamma$ is expanded by all relations with an
existential definition in $\Gamma$.

\ignore{
\begin{cor}\label{cor:exdefn}
Every reduct $\Gamma$ of the random graph has quantifier-elimination if it is expanded by all relations
with an existential definition in $\Gamma$.
\end{cor}
\begin{proof}
 This follows directly from the model-completeness of $\Gamma$ and the fact mentioned above that in
model-complete structures first-order formulas are equivalent to
existential formulas (see e.g.~\cite[Theorem~7.3.1]{Hodges}).
\end{proof}
}

As another application, we refine
Theorem~\ref{thm:reducts2} by giving a finer (at least in theory)
classification of the reducts of the random graph.

\begin{cor}\label{cor:classificationUpToExistential}
Up to existential interdefinability, the random graph has exactly
five reducts.
\end{cor}
\begin{proof}
In the same way as in the proof of Corollary~\ref{cor:mc}, we can
use Theorem~\ref{thm:endos} to show that either the self-embeddings
of a reduct $\Gamma$ are generated by the automorphisms, and
$\Gamma$ is existentially interdefinable with one of the structures
described in Theorem~\ref{thm:reducts}; or otherwise $\Gamma$ has an
existential definition in $(V;=)$, which is again one of the five
cases from Theorem~\ref{thm:reducts}.
\end{proof}

The endomorphism monoid $\End(G)$ of the random graph has been
studied in~\cite{DelicDolinka,BonatoDelic,RandomGraphMonoid}. By
Proposition~\ref{thm:pres}, studying closed transformation monoids
containing $\Aut(G)$ is equivalent to studying reducts of $G$ up to existential positive
interdefinability. A complete classification of all locally closed
transformation monoids that contain all permutations of $V$, and
hence of the reducts of $(V; =)$ up to existential positive
interdefinability, has been given in~\cite{BodChenPinsker}; there is
only a countable number of such monoids. The results of the present
paper are far from providing a full classification of the locally
closed transformation monoids that contain $\Aut(G)$ -- this is left for future investigation.

\section{Ramsey-theoretic preliminaries}\label{ssect:ramsey}
We recall some
Ramsey-type theorems and extend these theorems for our purposes. This will allow us to find patterns in colorings of edges and non-edges
of graphs and of graphs equipped with additional structure.

We start by recalling a theorem on ordered structures due to
 Ne\v set\v ril and R\"odl \cite{NesetrilRoedlOrderedStructures} of which we will make heavy use. Let
$\tau=\tau' \cup \{\prec\}$ be a relational signature, and let $\cal
C(\tau)$ be the class of all finite $\tau$-structures $\S$ where
$\prec$ denotes a linear order on the domain of $\S$. For
$\tau$-structures $\A, \B$, let ${\A}\choose{\B}$ be the set of all
substructures of $\A$ that are isomorphic to $\B$ (we also refer to
members of ${\A}\choose{\B}$ as \emph{copies of $\B$ in $\A$}). For
a finite number $k\geq 1$, a \emph{$k$-coloring} of the copies of
$\B$ in $\A$ is simply a mapping $\chi$ from ${{\A}\choose{\B}}$
into a
 set of size $k$.

\begin{definition}
For $\S,\H,\P \in \cal C(\tau)$ and $k \geq 1$, we write $\S
\rightarrow (\H)^\P_k$ iff for every $k$-coloring $\chi$ of the
copies of $\P$ in $\S$ there exists a copy $\H'$ of $\H$ in $\S$
such that all copies of $\P$ in $\H'$ have the same color under
$\chi$.
\end{definition}

\begin{theorem}[from~\cite{NesetrilRoedlOrderedStructures,
NesetrilRoedlPartite, AbramsonHarrington}]\label{thm:ramsey} The
class $\cal C(\tau)$ of all finite relational ordered
$\tau$-structures is a \emph{Ramsey class}, i.e., for all $\H,\P \in
\cal C(\tau)$ and $k \geq 1$ there exists $\S \in \cal C(\tau)$ such
that $\S \rightarrow (\H)^\P_k$.
\end{theorem}

\begin{cor}\label{cor:coloringEdgesAndNonEdgesOfG}
    For every finite graph $\H$ and for all colorings $\chi_E$ and $\chi_N$ of the edges and the non-edges of the random graph
    $G$, respectively, by finitely many colors, there exists an isomorphic copy of $\H$ in $G$ on which both colorings are constant.
\end{cor}
\begin{proof}
    Let $k$ be the number of colors used altogether by $\chi_E$ and $\chi_N$.
    Let $\prec$ be any total order on the domain of $\H$, and denote the structure obtained from $\H$ by adding the order $\prec$ to the signature
    by $\bar{\H}$. Consider the complete graph $\K_2$ on two vertices, and order its two vertices anyhow to arrive at a structure $\bar{\K_2}$. Then the
    coloring $\chi_E$ of the edges of $\H$ can be viewed as a
    coloring of the copies of $\bar{\K_2}$ in $\bar{\H}$. Let $\bar{\S}$ with $\bar{\S}\rightarrow (\bar{\H})^{\bar{\K_2}}_k$ be provided by
    the preceding theorem, and let $\S$ be $\bar{\S}$ without the order. Then $\S$ is a graph with the property that whenever
    we color its edges with $k$ colors, then there is a copy of $\H$
    in $\S$ all of whose edges have the same color. Now we repeat
    the argument for the non-edges, starting from $\S$ instead of
    $\H$. We then arrive at a graph $\T$ with the property that
    whenever we color its edges and non-edges by $k$ colors, then
    there is a copy $\H'$ of $\H$ in $\T$ such that all edges of $\H'$ have the same color,
    and such that
    non-edges of $\H'$ have the same color. $\T$ has a copy in $G$,
    proving the claim.
\end{proof}

We will not only need to color edges of graphs, but also of graphs
equipped with additional structure.

\begin{defn}
\hfill
\begin{itemize}
\item
    An \emph{$n$-partitioned graph} is a structure $\U=(U;F,U_1,\ldots,U_n)$, where $(U;F)$ is a graph and each
    $U_i$ is a subset of $U$ such that the $U_i$ form a partition of $U$.
    \item     An \emph{$n$-constant graph} is a structure $\U=(U;F,u_1,\ldots,u_n)$, where $\U=(U;F)$ is a graph, and $u_i\in U$ are distinct.
\end{itemize}
\end{defn}

Observe that $n$-constant graphs are not relational structures;
therefore, in order to apply Theorem~\ref{thm:ramsey}, we have to
make them relational: To every $n$-constant graph 
$\U=(U;F,u_1,\ldots,u_n)$ we can assign an $n+2^n$-partitioned graph 
$\tilde{\U}=(U;F,\{u_1\},\ldots,\{u_n\},U_1,\ldots,U_{2^n})$ in
which the $u_i$ belong to singleton sets, and in which for every
possible relative position (edge or non-edge) to the $u_i$ we have a
set $U_j$ of all elements in $U\setminus\{u_1,\ldots,u_n\}$ having
this position. (In the language of model theory, every of the
$n+2^n$ sets corresponds to a maximal quantifier-free $1$-type over
the structure $\U$.) We call the parts $U_i$ the \emph{proper} parts
of $\tilde{\U}$.

\begin{defn}
    Let $\Gamma$ be a structure and $a^1,\ldots,a^m\in \Gamma$. We write
    $\typ(a^1,\ldots,a^m)$ for the set of quantifier-free formulas satisfied by
    the tuple $(a^1,\ldots,a^m)$ in $\Gamma$, and refer to this set as the \emph{type} of $(a^1,\ldots,a^m)$ in $\Gamma$.
\end{defn}

\begin{defn}
    Let $\Gamma$ be a structure and let $m\geq 1$. A coloring $\chi$ of the $m$-element subsets of $\Gamma$ is called \emph{canonical} iff for all tuples $(a^1,\ldots,a^m)$ and $(b^1,\ldots,b^m)$ enumerating $m$-element subsets of $\Gamma$, if $\typ(a^1,\ldots,a^m)=\typ(b^1,\ldots,b^m)$, then they induce subsets of equal color under $\chi$.
\end{defn}

In this section, we will consider colorings of the two-element subsets of graphs, $n$-partitioned graphs and $n$-constant graphs. For disjoint subsets $S_1$ and $S_2$ of any such structure, we will say that a coloring is canonical on $S_1$ iff it satisfies the definition of canonicity for subsets of $S_1$; moreover, we will say that a coloring is canonical between $S_1$ and $S_2$ iff it satisfies the definition of canonicity for subsets of $S_1\cup S_2$ which have precisely one element in $S_1$ and one element in $S_2$.

\begin{lem}[The $n$-partitioned graph Ramsey
lemma]\label{lem:partitionedGraphRamseyLemma}
    Let $n, k\geq 1$. For any finite $n$-partitioned graph
    $\U=(U;F,U_1,\ldots,U_n)$ there exists a finite $n$-partitioned graph
    $\Q=(Q;D,Q_1,\ldots,Q_n)$ with the property that for all colorings of
    the two-element subsets of $Q$ with $k$ colors,
    there exists a copy of $\U$ in $\Q$ on which the coloring is canonical.
\end{lem}
\begin{proof}

We show the lemma for $n=2$; the generalization to larger $n$ is
straightforward. For $n=2$, we apply Theorem~\ref{thm:ramsey} six
times: Once for the edges in $U_1$, once for the edges in $U_2$,
once for the edges between $U_1$ and $U_2$, and then the same for
all three kinds of non-edges.

In general, we would have to apply the theorem $2\ (n+{n
\choose{2}})$ times: Once for the edges of each part $U_i$, once for
the edges between any two distinct parts $U_i, U_j$, and then the
same for all non-edges on and between parts.

So assume $n=2$. We exhibit the idea in detail for the edges between
$U_1$ and $U_2$. Let $\prec$ be any total order on $U$ with the
property that $u_1\prec u_2$ for all $u_1\in U_1$, $u_2\in U_2$.
Consider the 2-partitioned graph ${\mathscr
L}^1=(\{a,b\};\{(a,b),(b,a)\},\{a\},\{b\})$ and order its vertices
by setting $a\prec b$; so ${\mathscr L}^1$ consists of two adjacent
vertices which are ordered somehow, and which lie in different
parts. By Theorem~\ref{thm:ramsey}, there exists an ordered
partitioned graph $\Q^1=(Q^1;D^1,Q_1^1,Q_2^1,\prec)$ such that $\Q^1
\rightarrow (\U)^{{\mathscr L}^1}_k$.

Now, if we change the order on $\Q^1$ in such a way that $r\prec s$
for all $r\in Q^1_1$ and all $s\in Q^1_2$ and such that the order
within the parts $Q^1_1, Q^1_2$ remains unaltered, then the
statement $\Q^1 \rightarrow (\U)^{{\mathscr L}^1}_k$ still holds:
For, given a coloring of the copies of ${\mathscr L}^1$ with respect
to the new ordering, we obtain a coloring of (possibly fewer) copies
of ${\mathscr L}^1$ with respect to the old ordering. There, we
obtain a copy $\U'$ of $\U$ such that all copies of ${\mathscr L}^1$
in $\U'$ have the same color. But in this copy, by the choice of the
order on $\U$, we have that $r\prec s$ for all $r\in U_1'$ and all
$s\in U_2'$. Therefore, this copy is also a substructure of $\Q^1$
with respect to the new ordering.

Since we can change the ordering on $\Q^1$ in the way described
above, the colorings of the copies of ${\mathscr L}^1$ are just
colorings of those pairs $\{r,s\}$, with $r\in Q_1^1$ and $s\in
Q_2^1$, which are edges.

Now we repeat the process with the structure ${\mathscr
L}^2=(\{a,b\};\{(a,b),(b,a)\},\{a,b\},\emptyset)$, ordered again by
setting $a\prec b$, starting with $\Q^1$. We then obtain a structure
$\Q^2$; this step takes care of the edges which lie within $U_1$.
After that we proceed with ${\mathscr
L}^3=(\{a,b\};\{(a,b),(b,a)\},\emptyset,\{a,b\})$, thereby taking
care of the edges within $U_2$. We then apply
Theorem~\ref{thm:ramsey} three more times with the structures
${\mathscr L}^4=(\{a,b\};\emptyset,\{a\},\{b\})$, ${\mathscr
L}^5=(\{a,b\};\emptyset,\{a,b\},\emptyset)$, and ${\mathscr
L}^6=(\{a,b\};\emptyset,\emptyset,\{a,b\})$, in order to ensure
homogeneous non-edges.
\end{proof}

We now arrive at the goal of this section, namely the following
lemma, which we are going to apply to operations on the random graph
numerous times in the sections to come.

\begin{lem}[The $n$-constant graph Ramsey lemma]\label{lem:constantGraphRamseyLemma}
    Let $n, k\geq 1$. For any finite $n$-constant graph
    $\U=(U;F,u_1,\ldots,u_n)$ there exists a finite $n$-constant graph
    $\Q=(Q;D,q_1,\ldots,q_n)$ with the property that for all colorings of the two-element subsets of
    $Q$ with $k$ colors, there exists a copy of $\U$ in
    $\Q$ on which the coloring is canonical.
\end{lem}
\begin{proof}
    Let $\tilde{\U}:=(U;F,\{u_1\},\ldots,\{u_n\}, U_1,\ldots,U_{2^n})$ be the
    partitioned graph associated with $\U$. We would
    like to use the partitioned graph Ramsey lemma (Lemma~\ref{lem:partitionedGraphRamseyLemma}) in order to
    obtain $\Q$; but we want the singleton sets $\{u_i\}$ of the
    partition to remain singletons, which is not guaranteed by that lemma.

    So consider the $2^n$-partitioned graph
    $\R:=(U\setminus\{u_1,\ldots,u_n\};F, U_1,\ldots,U_{2^n})$, and apply the partitioned
    graph Ramsey lemma to this graph to obtain a partitioned graph $\R^0$.

    Equip $\R^0$ with any linear order. Now consider the
    ordered $2^n$-partitioned graph ${\mathscr L}^1$ which has just
    one vertex, and whose first part contains this single vertex.
    Apply Theorem~\ref{thm:ramsey} in order to obtain an ordered partitioned graph $\R^1$ such that
    $\R^1\To (\R^0)^{{\mathscr L}^1}_{k^n}$.

    Next, consider the ordered $2^n$-partitioned graph ${\mathscr L}^2$ which has just
    one vertex, and whose second part contains this single vertex. Apply Theorem~\ref{thm:ramsey} in order to obtain
    an ordered
    partitioned graph $\R^2$ such that
    $\R^2\To (\R^1)^{{\mathscr L}^2}_{k^n}$.

    Repeat this procedure with the ordered
    $2^n$-partitioned graphs ${\mathscr L}^3,\ldots,{\mathscr L}^{2^n}$; ${\mathscr L}^i$ has its single vertex in its $i$-th part. We end up with an ordered partitioned graph
    $\R^{2^n}$. We now forget its order and denote the resulting structure by $\T=(T;C,T_1,\ldots,T_{2^n})$.

    $\T$ has the following property: Whenever we color its vertices with
    $k^n$ colors, then we find a copy of $\R^0$ in $\T$ such that the coloring is constant on each part of this copy.
    Hence, it has the property
    that if we color its two-element subsets and its vertices with
    $k$ and $k^n$ colors, respectively, then we find in it a copy of
    $\R$ on which the first coloring is canonical,
    and such that the color of the vertices depends only on the part the vertex lies in.

    Now consider the structure
    $\S:=(T\cup\{u_1,\ldots,u_n\};B,\{u_1\},\ldots,\{u_n\},T_1,\ldots,T_{2^n})$, where $B$ consists of the edges of $\T$, plus
    edges connecting the $u_i$ with the vertices of some parts $T_i$, depending on whether $u_i$ was in $\U$ connected to the vertices in $U_i$ or not.
    Clearly, $\S$ is the partitioned graph of the $n$-constant graph $\Q:=(T\cup\{u_1,\ldots,u_n\};B,u_1,\ldots,u_n)$.
    We claim that $\Q$ has the property we want to prove.
    Assume that we color the two-element subsets of $T\cup\{u_1,\ldots,u_n\}$ with $k$
    colors. We must find a copy of $\U$ in $\Q$ on which the coloring
    is canonical. Divide the coloring into two colorings, namely the coloring
    restricted to two-element subsets of $T$, and the coloring of two-element subsets which contain at least one element $u_i$
    outside $T$. The color of the sets $\{u_i,u_j\}$ completely outside $T$ is irrelevant for what we want to prove, so forget about these.

    Now the coloring of those sets which have exactly one element outside $T$ can be encoded in a coloring of the vertices of
    $T$: Each vertex is given one of $k^n$ colors, depending on the colors of its edges leading to
    $u_1,\ldots,u_n$. So we have encoded the original coloring into a
    coloring of two-elements subsets of $T$ and a coloring of the vertices of $T$.
    With our observation above, this proves the lemma.
\end{proof}

\section{Finding structure in mappings on the random graph}\label{sect:structureInMappings}

In this section we show how to use the Ramsey-theoretic results from
the last section in our context. That is, we will use those results in order to find regular patterns in the behavior of unary functions from $V$ to $V$.

\begin{defn}
    Let $\tau$ be any signature and let $\cal C(\tau)$ be a class of finite $\tau$-structures. We say that a property $P$ \emph{holds for arbitrarily large
    elements of $\cal C(\tau)$} iff for any $\F\in\cal C(\tau)$ there exists $\H\in\cal C(\tau)$  such that $\F$ embeds into $\H$ and $P(\H)$
    holds. We say that \emph{$P$ holds for all sufficiently large
    elements of $\cal C(\tau)$} iff there exists $\F\in\cal C(\tau)$ such that $P$ holds for $\H$ whenever $\F$ embeds
    into $\H$.
\end{defn}

Our classes $C(\tau)$ will be closed under induced substructures; moreover, our properties $P$ will be \emph{hereditary}, i.e., if $P(\H)$ holds, then $P$
also holds for all substructures of $\H$. The definition then says
that $P$ holds for arbitrarily large elements of $\cal C(\tau)$ iff
for any $\F\in\cal C(\tau)$ there is $\F'\in\cal C(\tau)$ isomorphic
to $\F$ such that $P(\F')$ holds.  

In our situation, $C(\tau)$ will also have the \emph{joint embedding property (JEP)}, i.e.,  for any
    two structures in $\cal C(\tau)$ there exists a structure in $\C(\tau)$ that embeds both structures. We then have that if $P$ holds for all sufficiently large elements of $C(\tau)$, then it holds for arbitrarily large elements of $C(\tau)$. Observe also that under (JEP), if arbitrarily large structures in $\cal C(\tau)$
have one of finitely many properties, then one of those properties holds for
arbitrarily large elements of $\cal C(\tau)$.

\begin{defn}
    Let $e,f \colon V\To V$. We say that $e$ \emph{behaves as $f$ on $F \subseteq V$} iff there is
    an automorphism $\alpha$ of $G$ such that $f(x)=\alpha(e(x))$ for
    all $x \in F$. We say that \emph{$e$ interpolates $f$ modulo
    automorphisms} iff for every finite $F \subseteq V$ there is an
    automorphism $\beta$ of $G$ such that $e(\beta(x))$ behaves as $f$
    on $F$; so this is the case iff there exist automorphisms
    $\alpha,\beta$ such that $\alpha(e(\beta(x))=f(x)$ for all $x\in
    F$.
\end{defn}

Note that if $e$ interpolates $f$ modulo automorphisms, then it also
generates $f$.

\begin{defn}\label{def:canonical}
    Let $\U$ be a graph, an $n$-partitioned graph, or an $n$-constant graph. Any function $f\colon \U\To \U$ induces a coloring of the two-element subsets of $\U$ as follows: the color of a set $\{x,y\}$ is the type of $(f(x),f(y))$ with respect to the graph relation of $\U$. We say that $f$ is \emph{canonical} iff the coloring it induces is canonical.
\end{defn}

\begin{proposition}\label{prop:interpolation}
Let $e \colon V \rightarrow V$ be a mapping on the random graph $G$. Then $e$ is canonical on arbitrarily large subgraphs of $G$, 
and
interpolates either the identity, $e_E$, $e_N$, a constant function,
or $\flip$ modulo automorphisms.
\end{proposition}
\begin{proof}
We show that arbitrarily large finite subgraphs of $G$ have the
property that $e$ behaves on them like one of the operations of the
proposition. Since there are finitely many operations to choose
from, $e$ then behaves like one fixed operation $p$ from the list on
arbitrarily large finite subgraphs of $G$. By the
homogeneity of $G$, we can freely move finite graphs
around by automorphisms, proving that $e$ interpolates $p$.

So let $\F$ be any finite graph; we have to find a copy $\F'$ of
$\F$ in $G$ such that $e$ behaves like one of the mentioned
operations on this copy.

We color all pairs $\{x,y\}$ of distinct vertices of $G$
\begin{itemize}
\item by $1$ if $e(x)=e(y)$,
\item by $2$ if $(e(x),e(y)) \in E$,
\item by $3$ if $(e(x),e(y)) \in N$.
\end{itemize}

By Corollary~\ref{cor:coloringEdgesAndNonEdgesOfG} there exists a
copy $\F'$ of $\F$ in $G$ such that all edges and all non-edges of
$\F'$ have the same color $\chi_E$ and $\chi_N$, respectively. If
$(\chi_E,\chi_N)=(1,1)$, then $e$ behaves like the constant function
on $F'$. If $(\chi_E,\chi_N)=(2,3)$, then it behaves like the
identity, and if $(\chi_E,\chi_N)=(3,2)$, then $e$ behaves like $\flip$.
If $(\chi_E,\chi_N)=(2,2)$ or $(\chi_E,\chi_N)=(3,3)$, then $e$
behaves like $e_E$ or $e_N$, respectively. Finally, it is easy to
see that $(\chi_E,\chi_N)=(1,q)$ or $(\chi_E,\chi_N)=(q,1)$, where
$q\in\{2,3\}$, is impossible if $\F$ contains the two three-element
graphs with one and two edges, respectively.
\end{proof}

\begin{defn}
    Let $\tau$ be a signature and let $T$ be a theory in this language. We call a $\tau$-structure \emph{$\aleph_0$-universal} for $T$ iff it satisfies $T$ and embeds all finite models of $T$.
\end{defn}

\begin{lem}[The $n$-partite graph interpolation lemma]\label{lem:n-partite-interpolation}
        Let $\U=(U;C,U_1,\ldots,U_n)$ be an $\aleph_0$-universal partitioned graph, and let $f \colon U\To U$.
        Then every finite partitioned graph has a copy in $\U$
        on which $f$ is canonical.
\end{lem}
\begin{proof}
    This is immediate from the $n$-partitioned graph Ramsey lemma (Lemma~\ref{lem:partitionedGraphRamseyLemma}):
    Just like in the proof of Proposition~\ref{prop:interpolation}, we color the edges and non-edges of $\U$ according to what $f$
    does to them.
\end{proof}

\begin{lem}[The $n$-constant graph interpolation lemma]\label{lem:n-constant-interpolation}
        Let $\U=(U;C,u_1,\ldots,u_n)$ be an $\aleph_0$-universal $n$-constant graph, and let $f \colon U\To U$.
        Then every finite $n$-constant graph has a copy in $\U$
        on which $f$ is canonical.
\end{lem}
\begin{proof}
    This is immediate from the $n$-constant graph Ramsey lemma (Lemma~\ref{lem:constantGraphRamseyLemma}).
\end{proof}

\section{Unary functions}\label{sect:minimalUnary}

We now have the tools to settle the unary case: In this section, we will prove Theorem~\ref{thm:mainUnary} which characterizes the unary minimal functions, Theorem~\ref{thm:reducts} which lists the five closed supergroups of $\Aut(G)$, and Theorem~\ref{thm:endos} which states that any closed monoid containing $\Aut(G)$ either is generated by the group of its permutations, or contains $e_E$, $e_N$, or a constant function. We start by applying Lemma~\ref{lem:n-constant-interpolation} to prove

\begin{lem}\label{lem:eN}
    Let $e \colon V\To V$ be so that it preserves $N$ but not $E$. Then $e$ generates $e_N$. Dually, if $e \colon V\To V$ is so that it preserves $E$ but not $N$, then $e$ generates $e_E$.

\end{lem}
\begin{proof}
    We prove that for every finite subset $F$ of $V$,
    $e$ generates an operation which behaves like $e_N$ on $F$. We
    first claim that there are adjacent vertices $a,b\in V$ such that $(e(a),e(b)) \in N$.
     Since $e$ does not preserve $E$, there exist $u,v$ with
    $(u,v) \in E$ such that $(e(u),e(v)) \nin E$. If $(e(u),e(v)) \in N$, then we are done.
    If $e(u)=e(v)$, then choose $w$ such that $(w,u) \in E$ and $(w,v) \in N$. We have $(e(w),e(u))=(e(w),e(v)) \in N$, so $u,w$ prove the claim.

    Now, $\U:=(V;E,a,b)$ is an $\aleph_0$-universal 2-constant graph. Therefore, by
    Lemma~\ref{lem:n-constant-interpolation}, $e$ is canonical on
    arbitrarily large substructures of $\U$. Since $e$ preserves
    $N$, it is easy to see that if $e$ is canonical on a $2$-constant
    graph which is sufficiently large, then  $e$ must be injective; for
    example, if $e$ is canonical on a graph which contains the three-element graph with two edges, then
    $e$ cannot collapse any edges of that graph. Hence, $e$ is canonical
    and injective on arbitrarily large $2$-constant subgraphs of
    $\U$. Since $e$ preserves $N$, we have that for arbitrarily large substructures of
    $\U$, it behaves like the identity or like $e_N$ on and between
    the parts of these structures; in particular, it does not turn any non-edges into edges on and between the parts of these structures. Hence, for any finite $2$-constant subgraph of $\U$, by applying an automorphism of $G$ and then $e$, we can delete
    the edge between the two constants without turning any non-edge of that $2$-constant graph into an edge. But that means that starting from any finite graph, we
    can delete all edges by repeating this process, choosing any
    edge we want to get rid of in each step. This proves the lemma.
\end{proof}

\begin{lem}\label{lem:constant}
    If $e \colon V\To V$ preserves neither $E$ nor $N$ and is not injective, then $e$ generates a constant operation.
\end{lem}
\begin{proof}
Since $e$ is not injecitve, it collapses without loss of generality an edge (otherwise dualize). Since $e$ violates $N$, it either collapses a non-edge or sends some non-edge to an edge, which, with the help of an appropriate automorphism, can be collapsed by another application of $e$. Thus $e$ generates operations $g,h$ which collapse an edge and a non-edge, respectively. Having this, one sees that $e$ generates a constant function on each finite subset $F$ of $V$, by shifting $F$ around with automorphisms and applying $g$ and $h$ to collapse all points in $F$ to a single vertex.
\ignore{
    We must show that for any finite subset $F$ of $V$, $e$ generates an operation which is constant on $F$.

    Observe that $e$ generates operations $g,h$ which collapse an edge
    and a non-edge, respectively. To see this, note that since
    $e$ is not injective, it collapses an edge or a non-edge; say
    without loss of generality it collapses an edge, so we can set
    $g:=e$. If it also collapses a non-edge, then we are done.
    Otherwise, since $e$ violates $N$, it sends some non-edge to an
    edge, which, with the help of an appropriate automorphism, can
    be collapsed by another application of $e$.

    Having this, one proceeds inductively to collapse all the vertices of
    $F$, shifting $F$ around with automorphisms accordingly and applying $g$ and
    $h$. After at most $|F|$ steps, the whole of $F$ is collapsed to a
    single vertex.}
\end{proof}

The following proposition already identifies the five minimal functions of Theorem~\ref{thm:mainUnary}.

\begin{prop}\label{prop:endos-weak}
Let $\Gamma$ be a reduct of $G$. Then one
of the following cases applies.
\begin{enumerate}
\item $\Gamma$ has a constant endomorphism.
\item $\Gamma$ has $e_E$ as an endomorphism.
\item $\Gamma$ has $e_N$ as an endomorphism.
\item $\Gamma$ has $\flip$ as an automorphism.
\item $\Gamma$ has $\sw$ as an automorphism.
\item $\Aut(G)$ is dense in $\End(\Gamma)$.
\end{enumerate}
\end{prop}

\begin{proof}

If $\Gamma$ has an endomorphism $e$ which preserves $E$ but not $N$
or $N$ but not $E$, then we can refer to Lemma~\ref{lem:eN}. If all of its endomorphisms preserve both
$N$ and $E$, then $\Aut(G)$ is dense in $\End(\Gamma)$. We thus assume henceforth that $\Gamma$ has an endomorphism $e$
which violates both $E$ and $N$.

If $e$ is not injective, then it generates a constant operation, by
Lemma~\ref{lem:constant}. So suppose that $e$ is injective. Fix $(x,y) \in E$ such that $(e(x),e(y)) \in N$.

By Proposition~\ref{prop:interpolation}, $e$ is canonical on
arbitrarily large finite subgraphs of $G$. If $e$ interpolates $\flip$,
$e_E$, or $e_N$ modulo automorphisms, then we are done. So assume
this is not the case, i.e., there is a finite graph $\F_0$ with the
property that on all copies of $\F_0$ in $G$, $e$ does not behave
like any of these operations. Observe that $e$ then behaves like the
identity on arbitrarily large subgraphs of $G$. Moreover, this
assumption implies that if a finite subgraph $\F$ of $G$ is
sufficiently large (i.e., if it embeds $\F_0$), and $e$ is canonical
on $\F$, then $e$ behaves like the identity on $\F$.

We now make a series of observations which rule out bad behavior of
$e$ between subsets of the random graph, and which follow from our
assumptions of the preceding paragraph; the easily verifiable
details are left to the reader.

\begin{itemize}
\item If $e$ behaves like $\flip$ between the parts of arbitrarily large
finite 2-partitioned subgraphs of $G$, then it generates $\sw$.

\item If $e$ behaves like $e_N$ between the parts of arbitrarily
large finite 2-partitioned subgraphs of $G$, then it generates
$e_N$.

\item If $e$ behaves like $e_E$ between the parts
 of arbitrarily large finite 2-partitioned subgraphs of $G$, then
 it generates $e_E$.
\end{itemize}

We assume therefore that for sufficiently large finite
$2$-partitioned subgraphs of $G$, if $e$ is canonical on such a
graph, then $e$ behaves like the identity on and between the parts.

Now observe that $\Q:=(V;E,x,y)$ is an $\aleph_0$-universal
2-constant graph. Let $\F=(F;D,f_1,f_2)$ be any finite 2-constant
graph. By the $n$-constant interpolation lemma
(Lemma~\ref{lem:n-constant-interpolation}), there is a copy $\F'$ of
$\F$ in $\Q$ on which $e$ is canonical. By our assumption above, if
only $\F$ is sufficiently large, then being canonical on a proper part
$F_i'$ of the 6-partitioned graph
$\tilde{\F'}=(F';E,\{x\},\{y\},F_1',\ldots,F_4')$ corresponding to
$\F'$ means behaving like the identity thereon, and being canonical
between proper parts means behaving like the identity between these
parts. Therefore, all 2-constant graphs $\F$ have a copy
$\F'=(F';E,x,y)$ in $\Q$ such that $e$ behaves like the identity on
and between all of the parts $F_i', F_j'$ of the corresponding
partitioned graph $\tilde{\F'}=(F';E,\{x\},\{y\},F'_1,\ldots,F'_4)$.

Of a two-constant graph $\F$, consider the reduct $\H=(F;D,f_1)$.
This reduct has a copy $\H'$ in $\Q^x=(V;E,x)$ on which $e$ is
canonical. The corresponding partitioned graph has two parts $H_1'$,
$H_2'$, and $x$ is adjacent to, say, all vertices in $H_1'$ and to
none in $H_2'$. Since $e$ is canonical on $\H'$, either it preserves all edges between $x$ and $H_1'$, or it turns all these edges into non-edges. Similarly with the non-edges
between $x$ and $H_2'$. If all edges are deleted and all non-edges
kept for arbitrarily large $\H$, then $e$ generates $e_N$. If all
edges are deleted and all non-edges edged for arbitrarily large
$\H$, then $e$ interpolates $\sw$ modulo automorphisms. If all edges
are kept and all non-edges edged for arbitrarily large $\H$, then
$e$ generates $e_E$. So we assume that if only $\H$ is sufficiently large,
then all edges and non-edges are kept by $e$ on those copies of $\H$
on which $e$ is canonical.

We use the same argument with the reduct $(F;D,f_2)$ and
$\Q^y=(V;E,y)$, and arrive at the conclusion that if the
two-constant graph $\F$ is sufficiently large, then on every copy of $\F$
in $\Q$ on which $e$ is canonical, the edges and non-edges leading
from $x$ and $y$ to the other vertices of the copy are kept.

Combining this with what we have established before, we conclude
that if only $\F$ is sufficiently large, and $\F'$ is a copy of $\F$ in
$\Q$ on which $e$ is canonical, then $e$ behaves like the identity
on $\F'$ except between $x$ and $y$, where it deletes the edge.
Hence, for any finite $\F$ we can find a copy in $\Q$ on which $e$
behaves that way. But this implies that starting from any finite
graph $\S:=(F;D)$, we can pick any edge in $\S$, say between
vertices $f_1,f_2$, and then find a copy of $\F:=(F;D,f_1,f_2)$ in
$\Q$ such that $e$ deletes exactly that edge from the copy whithout
changing the rest. Hence, by shifting finite graphs around with
automorphisms, we can delete a single edge from an arbitrary finite
subgraph of $G$ without changing the rest of the graph. Applying
this successively, we can remove all edges from arbitrary finite
graphs, proving that $e$ generates $e_N$.
\end{proof}

Now Theorem~\ref{thm:mainUnary} follows: let $f$ be a minimal function and $\Gamma$ the reduct whose endomorphism monoid is generated by $f$. 
We apply Proposition~\ref{prop:endos-weak} to $\Gamma$. Observe that Case~(6) of that proposition cannot hold for $\Gamma$, since its endomorphism $f$ is non-trivial, and hence not generated $\Aut(G)$. Thus, $\Gamma$ contains one of the functions of the other cases, meaning that $f$ is equivalent to one of those functions. This finishes the proof.

\begin{proposition}\label{prop:above-minus}
Let $\Gamma$ be a reduct of $G$, and
suppose $\Gamma$ is preserved by $\flip$, but not by $e_N, e_E$, or a
constant operation. Then the endomorphisms of $\Gamma$ are
generated by $\{\flip\} \cup \Aut(G)$, or $\Gamma$ is preserved by $\sw$.
\end{proposition}
\begin{proof}

Suppose the endomorphisms of $\Gamma$ are not generated by $\{\flip\}
\cup \text{Aut}(G)$. Then, by Proposition~\ref{prop:galois}, there
is a relation $R$ invariant under $\{\flip\} \cup \text{Aut}(G)$ and an
endomorphism $e$ of $\Gamma$ which violates $R$; that is, there
exists a tuple $a:=(a_1,\ldots,a_n)\in R$ such that
$e(a)=(e(a_1),\ldots,e(a_n))\nin R$.

Since $R$ is definable in $G$, $e$ violates either an
edge or a non-edge. Hence, as in the proof of
Proposition~\ref{prop:endos-weak}, the assumption that $e$ does not
generate $e_N$, $e_E$, or a constant operation implies that $e$ is
injective.

Let $\F=(F;D,f_1,\ldots,f_n)$ be any finite $n$-constant graph. By
the $n$-constant interpolation lemma
(Lemma~\ref{lem:n-constant-interpolation}), there is a copy $\F'$ of
$\F$ in the $\aleph_0$-universal $n$-constant graph
$\Q:=(V;E,a_1,\ldots,a_n)$ such that $e$ is canonical on this copy.

We now make a series of observations on the behavior of $e$ on and
between subsets of $V$ where it is canonical.

\begin{itemize}

\item Since by assumption, $e$ does not
interpolate $e_E$, $e_N$, or a constant operation modulo
automorphisms, it behaves like $\flip$ or the identity on sufficiently
large finite subgraphs of $G$ where it is canonical.

\item Suppose that for arbitrarily large finite $2$-partitioned subgraphs of
$G$, $e$ behaves like the identity on the parts and like $\flip$ between
the parts. Then $e$ generates $\sw$.

\item Suppose that for arbitrarily large finite $2$-partitioned subgraphs of
$G$, $e$ behaves like the identity on the parts and like $e_N$ (like
$e_E$) between the parts. Then $e$ generates $e_N$ ($e_E$).

\item Suppose that for arbitrarily large finite $2$-partitioned subgraphs of
$G$, $e$ behaves like $\flip$ on the parts and like the identity / $e_N$
/ $e_E$ between the parts. Then $e$ and $\flip$  together generate $\sw$
/ $e_E$ / $e_N$. This is because we can apply the preceding two
observations to $-e$.

\item Suppose that for arbitrarily large finite $2$-partitioned subgraphs of
$G$ on which $e$ is canonical, $e$ behaves like $\flip$  on one part and
like the identity on the other part. Then $e$ and $\flip$  together
generate $e_N$.

\end{itemize}

To see the last assertion for the case where $e$ behaves like the
identity between the parts, select an edge within one of the parts
that is mapped to a non-edge. For arbitrary finite $A \subseteq V$
we can now use the operation $e$ to get rid of one edge in the graph
induced by $A$ in $G$ and preserve all other edges, and so
eventually generate an operation that behaves like $e_N$ on $A$. For
the case where $e$ behaves like $\flip$  between the parts, we can apply
the same argument to $-e$. If $e$ behaves like $e_N$ between the
parts, then we can all the more delete edges. If it behaves like
$e_E$ between the parts, then $-e$ behaves like $e_N$ and we are
back in the preceding case.

Summarizing our observations, we can assume that for an arbitrary
finite $n$-constant graph $\F$ there is a copy of $\F$ in $\Q$ such
that $e$ behaves like the identity on and between all proper parts
$F_i',F_j'$ of the corresponding partitioned graph, or like $\flip$  on
and between all of its parts. If only the second case holds for
arbitrarily large $n$-constant graphs $\F$, then we simply proceed
our argument with $-e$ instead of $e$. We can do that since also
$-e(a)\nin R$. Otherwise, picking an automorphism $\alpha$ of
$G$ such that $\alpha (-(-x))=x$ for all $x\in V$, we would have
$\alpha(-(-e(a)))=e(a)\in R$, contrary to our choice of $a$. Thus we
assume that for arbitrary finite $n$-constant graphs $\F$ there is a
copy of $\F$ in $\Q$ such that $e$ behaves like the identity on and
between all proper parts of that copy.

As in the proof of Proposition~\ref{prop:endos-weak}, we may assume that
if a copy $\F'=(F';E,a_1,\ldots,a_n)$ of $\F$ in $\Q$ is large
enough and $e$ is canonical on $\F'$ and behaves like the identity
on and between all proper parts $F_i',F_j'$ of the corresponding
$n$-partitioned graph $\tilde{\F'}$, then it leaves the edges and
non-edges between the $a_i$ and the vertices in
$F'\setminus\{a_1,\ldots,a_n\}$ unaltered. It follows that for
arbitrary finite $n$-constant graphs $\F$ there is a copy of $\F$ in
$\Q$ such that the only edges or non-edges changed by $e$ on this
copy are those between the $a_i$.

Finally, note that since $R$ is definable in the random graph and
$e(a)\nin R$, $e$ destroys at least one edge or one non-edge on
$\{a_1,\ldots,a_n\}$. Without loss of generality, say that $a_1,
a_2$ are adjacent but their values under $e$ are not. We have shown
that for arbitrarily large $2$-constant graphs $\H$, there is a copy
of $\H$ in $(V;E,a_1,a_2)$ such that $e$ behaves like the identity
on this copy, except for the edge between $a_1$ and $a_2$, which is
destroyed. This clearly implies that $e$ generates $e_N$.
\end{proof}

\begin{proposition}\label{prop:above-switch}
Let $\Gamma$ be a reduct of $G$, and
suppose $\Gamma$ is preserved by $\sw$, but not by $e_N, e_E$, or
a constant operation. Then the endomorphisms of $\Gamma$ are
generated by $\{\sw\} \cup \Aut(G)$, or $\Gamma$ is preserved
by $\flip$.
\end{proposition}
\begin{proof}

The proof is very similar to the proof of the preceding proposition.
 This time we know that unless the endomorphisms
are generated by $\{\sw\} \cup \Aut(G)$, there exists
an endomorphism $e$ that violates a relation $R$ which is preserved
by $\{\sw\} \cup \Aut(G)$. Fix a tuple $a$ as before.

As in the preceding proof, we may assume that $e$ is injective. If
$e$ interpolates $\flip$  modulo automorphisms, we are done. Suppose
therefore that if $e$ is canonical on a finite partitioned graph
which is sufficiently large, then it must behave like the identity on its parts.

If $e$ behaves like $e_N$ ($e_E$) between the parts of arbitrarily
large finite $2$-partitioned subgraphs of $G$, then it generates
$e_N$ ($e_E$). Thus we may assume that it behaves like the identity
or $\flip$  between such parts.

Suppose that for arbitrarily large finite $3$-partitioned subgraphs
$\F=(F;E,F_1,F_2,F_3)$ of $G$ on which $e$ is canonical, $e$ behaves
like the $\flip$  between exactly two of the parts, say between
$F_1,F_2$, and like the identity between $F_2,F_3$ and $F_1,F_3$.
Then $e$ is easily seen to generate both $e_N$ and $e_E$. Indeed, if
we want to delete\footnote{For the purposes of the proof, we
identify ourselves with the endomorphism monoid.} any
edge from a finite graph, then we can view the vertices of the edge
as two parts of a $3$-partitioned graph, where the third part
contains all the other vertices. If $e$ behaves like $\flip$  between the
two vertices whose edge we want to delete, and like the identity on
and between the other parts, what happens is exactly that the edge
is deleted.

If for arbitrarily large finite $3$-partitioned subgraphs $\F$ of
$G$ on which $e$ is canonical, $e$ behaves like $\flip$  between, say,
$F_1,F_2$ and $F_1,F_3$, and like the identity between $F_2,F_3$,
then by applying a suitable switch operation $i_A$ to $e$ we are
back in the preceding case. Note here that there is an automorphism
$\alpha$ of $G$ such that $i_A(\alpha(i_A(x)))=x$ for all $x\in V$.
Therefore, $i_A(e(a))\nin R$; for otherwise, we would have
$i_A(\alpha(i_A(e(a)))=e(a)\in R$, a contradiction.

The latter argument works also if $e$ behaves like $\flip$  between all
three parts. Summarizing, we may assume that if $e$ is canonical on
a finite $n$-partitioned graph which is sufficiently large, where $n\geq
3$, then it behaves like the identity on and between all of the
parts.

As for $n$-constant graphs on which $e$ is canonical, $e$ might flip
edges and non-edges between some parts and the constants. However,
this situation can easily be repaired by a single application of
$\sw$.

Finally, observe that at least one edge or one non-edge on
$a_1,\ldots,a_n$ is destroyed, and that we therefore can generate
either $e_N$ or $e_E$.
\end{proof}

\begin{proposition}\label{prop:to-all-permutations}
Let $\Gamma$ be a reduct of $G$, and
suppose $\Gamma$ is preserved by $\sw$ and by $\flip$, but not by $e_N,
e_E$, or a constant operation. Then the endomorphisms of $\Gamma$
are generated by $\{\flip,\sw\} \cup \Aut(G)$, or $\Gamma$
is preserved by all permutations.
\end{proposition}
\begin{proof}
The argument goes as in the preceding two propositions; we leave the
details to the reader.
\end{proof}

Theorem~\ref{thm:endos} now is a direct consequence of
Proposition~\ref{prop:endos-weak}, and
Propositions~\ref{prop:above-minus}, \ref{prop:above-switch},
\ref{prop:to-all-permutations}: If a reduct $\Gamma$ of $G$ does not
have $e_E$, $e_N$, or a constant operation as an endomorphism, and
if its endomorphisms are not generated by $\Aut(G)$,
then Proposition~\ref{prop:endos-weak} implies that it has either $\flip$  or
$\sw$ as an endomorphism. Since $\Aut(\Gamma)$ contains $\Aut(G)$,
once $\Gamma$ has $\flip$  or $\sw$ as an endomorphism, it also has its
inverse as an endomorphism; thus it has $\flip$  or $\sw$ as an
automorphism. But then by the preceding three propositions, either
$\End(\Gamma)$ is generated by $\Aut(\Gamma)$, or $\Gamma$ is
preserved by all permutations. The latter case, however, is
impossible, as this would imply that $e_E$ and $e_N$ are among its
endomorphisms, which we excluded already.

Observe also how Thomas' classification of closed permutation groups
containing $\Aut(G)$ (Theorem~\ref{thm:reducts}) follows from our
results: If a closed group properly contains $\Aut(G)$, then it contains
$\flip$  or $\sw$, by Proposition~\ref{prop:endos-weak}. If it contains $\flip$ 
but is not generated by $\flip$ , then it contains $\sw$ by
Proposition~\ref{prop:above-minus}. Similarly, if it contains $\sw$
but is not generated by $\sw$, then it contains $\flip$  by
Proposition~\ref{prop:above-switch}. If it contains both $\flip$  and
$\sw$, but is not generated by these operations, then it must
already contain all permutations
(Proposition~\ref{prop:to-all-permutations}).

\section{Producing binary injections}\label{sect:reducingArity}

Having found the minimal unary operations, we now turn to operations of higher arity. The goal of this section is proving Theorem~\ref{thm:bin-inj}, which (together with Lemma~\ref{lem:en})
implies that all minimal functions are at most binary.

\begin{defn}
We say that an operation $f \colon V^k \rightarrow V$
is \emph{essentially unary} iff there exists a unary function $g \colon  V \rightarrow V$
and $1\leq i \leq k$ such that $f(x_1,\dots,x_k)=g(x_i)$ for all
$x_1,\dots,x_k \in V$.
If $f$ is not essentially unary, we call it
\emph{essential}.
\end{defn}

\begin{theorem}\label{thm:bin-inj}
Let $f$ be an essential operation on the random graph that preserves
$E$ and $N$.
Then $f$ generates a binary injection.
\end{theorem}


\begin{lem}\label{lem:en}
    Minimal essential operations on the random graph must preserve $E$ and $N$.
\end{lem}
\begin{proof}
    Suppose an essential function $f \colon V^n\To V$ does not preserve $E$ (the argument for $N$ is dual). Then there exist tuples $(x_1,\ldots,x_n), (y_1,\ldots,y_n)$ such that $(x_i,y_i) \in E$ for all $1\leq i\leq n$ and such that $f(x_1,\ldots,x_n)$ and $f(y_1,\ldots,y_n)$ are not connected by an edge. Fix $(u,v) \in E$, and choose automorphisms $\alpha_i$ such that $\alpha_i(u)=x_i$ and $\alpha_i(v)=y_i$, for all $1\leq i\leq n$. The function $g(x):=f(\alpha_1(x),\ldots,\alpha_n(x))$ is unary and violates $E$; hence it is non-trivial. But being unary, it cannot generate the essential function $f$, proving that $f$ is not minimal.
\end{proof}

The following lemma allows us to work with binary operations; its proof is very similar to the proof of a corresponding lemma in~\cite{ecsps}.

\begin{lem}\label{lem:binaryPoly}
Let $f \colon V^k \rightarrow V$ be an essential operation. Then
$f$ generates a binary essential operation.
\end{lem}
\begin{proof}
Assume without loss of generality that $f$ depends all of its arguments
and is at least ternary. In particular, there are $a_1,\dots,a_k$ and $a_1'$ such that
$f(a_1,\dots,a_k) \neq f(a'_1,a_2,\dots,a_k)$.
We distinguish two cases.

\textbf{Case 1.} There are $b_1,\dots,b_k$ such that
$(b_i,a_i) \in E$ for $2 \leq i \leq k$ and $f(b_1,a_2,\dots,a_k) \neq
f(b_1,\dots,b_k)$. By the homogeneity of $G$ we can find
automorphisms $\alpha_3,\dots,\alpha_k$ such that
$\alpha_i(a_2)=a_i$ and $\alpha_i(b_2) =b_i$. Using these
automorphisms we define
$$g(x,y):=
f(x,y,\alpha_3(y),\dots,\alpha_k(y))\; ,$$
which clearly depends on both
arguments.

\textbf{Case 2.} For all $b_1,\dots,b_k$, if $(a_i,b_i) \in E$ for $2
\leq i \leq k$, then $f(b_1,a_2,\dots,a_k)=f(b_1,b_2,\dots,b_k)$.
Since $f$ depends on its second coordinate, there are
$c_1,\dots,c_k$ and $c_2'$ such that
$$f(c_1,c_2,c_3,\dots,c_k) \neq f(c_1,c_2',c_3,\dots,c_k)\; .$$
Then $f(c_1,a_2,\dots,a_k)$ can be
equal to either $f(c_1,c_2,c_3,\dots,c_k)$, or to
$f(c_1,c_2',c_3,\dots,c_k)$, but not to both. We assume without loss
of generality that $f(c_1,a_2,\dots,a_k) \neq
f(c_1,c_2,c_3,\dots,c_k)$. From the extension property of the random graph
we see that we can choose $d_2,\dots,d_k$ such that $(d_i,a_i) \in E$ and
$(d_i,c_i) \in E$ for $2 \leq i \leq k$. Since $G$ is homogeneous
there are automorphisms $\alpha_3,\dots,\alpha_k$ of $G$ such that
$\alpha_i(c_2)=c_i$ and $\alpha_i(d_2)=d_i$. We claim that the
operation $g$ defined by
$$ g(x,y) := f(x,y,\alpha_3(y),\dots,\alpha_k(y))$$
depends on both arguments.
Indeed, we know
that $g(a_1,d_2)=f(a_1,d_2,\dots,d_k)=f(a_1,\dots,a_k)$, and that
$f(a'_1,d_2)=f(a_1',d_2,\dots,d_k)=f(a'_1,a_2,\dots,a_k)$. By the
choice of the values $a_1,\dots,a_k$ and $a_1'$ these two values
are distinct, and we have that $g$ depends on the first argument.
For the second argument, note that $g(c_1,d_2) = f(c_1,d_2,\dots,d_k) = f(c_1,a_2,\dots,a_k)$ and that $g(c_1,c_2) = f(c_1,c_2,\dots,c_k)$.
Because $f(c_1,a_2,\dots,a_k)$ and $f(c_1,c_2,\dots,c_k)$ are distinct,
we have that $g$ also depends on the second argument.
\end{proof}

We are left with the task of producing a binary injection from a binary essential function preserving $E$ and $N$. This will be prepared in the general Lemma~\ref{lem:intersect}.

\begin{definition}
A relation $R\subseteq X^k$ is called \emph{intersection-closed}
iff for all $(u_1,\ldots,u_k), (v_1,\ldots,v_k) \in R$
there is a tuple $(w_1,\ldots,w_k) \in R$
such that for all $1 \leq i,j \leq k$ we have
$w_i \neq w_j$ whenever $u_i \neq u_j$ or $v_i \neq v_j$.
\end{definition}

\begin{lem}\label{lem:intersect}
Let $\Gamma$ be a countable $\omega$-categorical structure
where $\neq$ is primitive positive definable.
Then the following are equivalent.
\begin{enumerate}
\item If $\phi$ is a primitive positive formula such that
both $\phi \wedge x \neq y$ and $\phi \wedge u \neq v$ are
satisfiable over $\Gamma$, then $\phi \wedge x \neq y \wedge u \neq
v$ is satisfiable over $\Gamma$ as well.
\item Every finite induced substructure of $\Gamma^2$ admits an
injective homomorphism into $\Gamma$.
\item $\Gamma$ is preserved by a binary injective operation.
\item All primitive positive definable relations in $\Gamma$ are intersection-closed.
\end{enumerate}
\end{lem}

We would like to remark that the first item in Lemma~\ref{lem:intersect} is inspired from joint work of the alphabetically first author with
Peter Jonsson and Timo von Oertzen in~\cite{HornOrFull}.

\begin{proof}
Throughout the proof,
let  $e_1,e_2, \dots$ be an enumeration
of the domain $D$ of $\Gamma$.
If $f$ is a binary injective polymorphism of $\Gamma$,
then clearly every relation in $\Gamma$ is intersection-closed,
so (3) implies (4).
The implication from (4) to (1) is straightforward as well.

We now show the implication from (1) to (2). Let $S$ be a finite induced
substructure of $\Gamma^2$. Without loss of generality we can
assume that $S$ is induced in $\Gamma^2$ by a set of the form
$\{e_1,\dots,e_n\}^2$, for sufficiently large $n$.
Consider the formula $\phi$ whose variables $x_1,\dots,x_{n^2}$ are the elements of $S$,
$$x_1 := (e_1,e_1),\dots, x_n := (e_1,e_n), \dots, x_{n^2-n+1} := (e_n,e_1) ,\dots, x_{n^2} :=(e_n,e_n) \; ,$$
and which is the conjunction over all
literals $R((e_{i_1},e_{j_1}),\dots,(e_{i_k},e_{j_k}))$ such that
$R(e_{i_1},\dots,e_{i_k})$ and $R(e_{j_1},\dots,e_{j_k})$ hold in $\Gamma$. So $\phi$ states precisely which relations hold in $S$.

Using induction
over the number of inequalities, we will now show that for any
conjunction $\sigma:=\bigwedge_{1\leq k\leq m} x_{i_k}\neq
x_{j_k}$ with the property that $i_k\neq j_k$ for all $1\leq k\leq
m$, the formula $\phi\wedge \sigma$ is satisfiable over $\Gamma$. This implies that there exists an $n^2$-tuple $t$ in $\Gamma$ with pairwise distinct entries
which satisfies $\phi$; the assignment that sends every $x_i\in S$ to $t_i$ is an injective homomorphism from $S$ into $\Gamma$.

For the induction beginning, let $x_i\neq x_j$ be any inequality. Let $r,s$ be the $n^2$-tuples defined as follows.
\begin{align*}
r & := (e_1,\dots,e_1,e_2,\dots,e_2,\dots,e_n,\dots,e_n) \\
s & := (e_1,e_2,\dots,e_n,e_1,e_2,\dots,e_n,\dots,e_1,e_2,\dots,e_n).
\end{align*}
These two tuples satisfy $\phi$, because the projections to
the first and second coordinate, respectively, are homomorphisms
from $S$ to $\Gamma$. Now either $r$ or $s$ satisfies $x_i\neq x_j$, proving that $\phi\wedge x_i\neq x_j$ is satisfiable in $\Gamma$.

In the induction step, let a conjunction $\sigma:=\bigwedge_{1\leq k\leq m} x_{i_k}\neq
x_{j_k}$ be given, where $m\geq 2$. Set $\sigma':=\bigwedge_{3\leq k\leq m} x_{i_k}\neq
x_{j_k}$, and $\phi':=\phi\wedge \sigma'$. Observe that $\phi'$ has a primitive positive definition in $\Gamma$, as $\phi$ and $\neq$ have such definitions. By induction hypothesis, both $\phi'\wedge x_{i_1}\neq x_{j_1}$ and $\phi'\wedge x_{i_2}\neq x_{j_2}$ are satisfiable in $\Gamma$. But then $\phi'\wedge x_{i_1}\neq x_{j_1}\wedge x_{i_2}\neq x_{j_2}=\phi\wedge\sigma$ is satisfiable over $\Gamma$ as well by (1), concluding the proof.

The implication from (2) to (3) is by a standard application of
K\"onig's lemma. This is because the fact that $\Gamma$ is $\omega$-categorical implies that for every $n\geq 1$, there are only
finitely many ``behaviors'' of functions from $\{e_1,\ldots,e_n\}^2$ to $\Gamma$, by the theorem of Ryll-Nardzewski.

We give the standard argument for completeness.
We say that two homomorphisms $f_1, f_2$ from the structure induced by a set $\{e_1,\dots,e_l\}^2$ in $\Gamma$ to $\Gamma$ are \emph{equivalent} if there is an automorphism
$h$ of $\Gamma$ such that $h(f_1(x,y))=f_2(x,y)$ for all $x,y \in \{e_1,\dots,e_l\}$.
Consider the infinite tree
$T$ whose vertices are the equivalence classes of injective homomorphisms from structures induced by a set of the form
$\{e_1, \dots, e_l\}^2$ to $\Gamma$.
There is an arc from one
equivalence class of injective
homomorphisms to another in $T$ iff there are
representatives $f_1$ and $f_2$ of the two classes such that the
domain of $f_1$ is $\{e_1,\dots,e_l\}^2$, and the domain of
$f_2$ is $\{e_1,\dots,e_l,e_{l+1}\}^2$, and $f_2$ is an extension of $f_1$.

The theorem of Ryll-Nardzewski implies that every node in $T$ has a finite number of
outgoing arcs, since there are only finitely many inequivalent homomorphisms from a set $\{e_1,\ldots,e_l\}^2$ to $\Gamma$.
Since $T$ is infinite by (2), K\"onig's lemma asserts the existence of an
infinite branch $B$ in $T$.
This infinite branch gives rise to an injective binary polymorphism $f$ of $\Gamma$, which is defined inductively as follows. The restriction of $f$ to $\{e_1,\dots,e_n\}$
will be an element from the $n$-th node of $B$. To start the induction, pick any function $f_1$ from the first node of $B$; $f_1$ has domain $\{e_1\}^2$. Set $f$ to equal $f_1$ on $\{e_1\}^2$.
Suppose $f$ is already defined on $e_1,\dots,e_n$, for $n\geq 1$. By
definition of $T$, we find representatives $f_n$ and
$f_{n+1}$ of the $n$-th and the $n+1$-st element of $B$ such
that $f_n$ is a restriction of $f_{n+1}$. The inductive assumption
gives us an automorphism $h$ of $\Gamma$ such that $h(f_n(x,y))=f(x,y)$
for all $x,y\in \{e_1,\dots,e_n\}$. We set $f(x,y)$ to be
$h(f_{n+1}(x,y))$, for all $x,y\in \{e_1,\dots,e_{n+1}\}$. The restriction of $f$ to $e_1,\dots,e_{n+1}$
will therefore be a member of the $n+1$-st node of $B$.
The operation $f$ defined in this way is indeed an injective homomorphism from $\Gamma^2$ to $\Gamma$, and we are done.
\end{proof}

We are now ready to end this section and provide a proof of Theorem~\ref{thm:bin-inj}.

\begin{proof}[Proof of Theorem~\ref{thm:bin-inj}]
Let an essential operation $f \colon V^k\To V$ preserving $E$ and $N$ be given. By Lemma~\ref{lem:binaryPoly}, $f$ generates a binary essential function; clearly, this function still preserves $E$ and $N$, so that we may henceforth assume that $f$ is itself binary.

Consider the structure $\Delta$ whose relations are the relations that are first-order definable in $G$ and preserved by $f$. In order to prove that $f$ generates a
binary injection, we will refer to Proposition~\ref{prop:galois} and prove that there is a binary injection
preserving $\Delta$.

By its definition, $\Delta$ has $E$ and $N$ amongst its relations. We claim that $\neq$ is also among the relations of $\Delta$: This is because
 $x\neq y $ iff
$\exists z \, (E(x,z) \wedge N(y,z))$, so $\neq$ has a primitive positive definition
from $E$ and $N$, and hence from $\Delta$. Hence, we may apply
Lemma~\ref{lem:intersect} to $\Delta$, and in order
to show that $\Delta$ is preserved by a binary injection, it
suffices to show that if $\phi$ is a primitive positive formula over
$\Delta$ such that both $\phi \wedge x \neq y$ and $\phi \wedge s
\neq t$ are satisfiable over $\Delta$, then $\phi \wedge x \neq y
\wedge s \neq t$ is satisfiable over $\Delta$ as well.

To this end, let $\phi$ be a primitive positive formula
over the signature of $\Delta$ such that
\begin{itemize}
\item there is a tuple $t_1$ that satisfies
$\phi \wedge x \neq y$
\item there is a tuple $t_2$ that satisfies
$\phi \wedge s \neq t$.
\end{itemize}
Let $a_1,a_2,a_3,a_4$ and $b_1,b_2,b_3,b_4$ be the values for
$x,y,s,t$ in $t_1$ and $t_2$, respectively. We have $a_1\neq a_2$
and $b_3\neq b_4$. We want to show that $\phi \wedge x \neq y \wedge
s \neq t$ is satisfiable over $\Delta$. Thus, if $a_3\neq a_4$ or
$b_1\neq b_2$, there is nothing to show, and so we assume that
$a_3=a_4$ and $b_1= b_2$.

We claim that there are automorphisms $\alpha,\beta$ of $G$ such
that in the tuple $t_3 := f(\alpha(t_1),\beta(t_2))$ the value of
$x$ is different from the value of $y$, and the value of $s$ is
different from the value of $t$. Then, since $f$ preserves $\Delta$,
the tuple $t_3$ shows that $\phi \wedge x \neq y \wedge s \neq t$ is
satisfiable over $\Delta$, and concludes the proof.

To prove the claim, we will find tuples $c:=(c_1,c_2,c_3,c_4)$ and
$d:=(d_1,d_2,d_3,d_4)$ of the same type as $(a_1,a_2,a_3,a_4)$ and
$(b_1,b_2,b_3,b_4)$, respectively, such that the tuple $e:=f(c,d)$
satisfies $e_1\neq e_2$ and $e_3\neq e_4$. Then, by the homogeneity
of $G$, we can find automorphisms $\alpha$ and $\beta$ of $G$
sending $a$ to $c$ and $b$ to $d$, which suffices for the prove of
our claim.

In the sequel, we will assume that $a_1,a_2\in X$ and $b_3,b_4\in Y$,
where $X,Y\in\{E,N\}$.\\

\textbf{Case 1.} Suppose first that $a_3=a_4\in\{a_1,a_2\}$ and
$b_1=b_2\in\{b_3,b_4\}$; without loss of generality $a_3=a_2$ and $b_1=b_3$.

\textbf{Case 1.1} There exists $u\in V$ such that for all $p,v\in V$
with $(u,v) \in Y$ we have $f(p,u)=f(p,v)$. Then, because $f$ preserves
$\neq$, we have $f(p,u) \neq f(q,u)$ for all $p\neq q$. Since $f$ is
essential there are $p,v\in V$ such that $f(p,u)\neq f(p,v)$: for if $f(p,u)=f(p,v)$ for all $p,v\in V$, then $f(p,v)=f(p,v')$ for all $p,v,v'\in V$, and hence $f$ is not essential. 
Pick
$w\in V$ such that $(w,u), (w,v) \in Y$. Pick moreover $q\in V$ such that
$(p,q)\in X$. We have $f(p,v)\neq f(p,u)=f(p,w)$. Moreover,
$f(p,w)=f(p,u)\neq f(q,u)=f(q,w)$. Hence, the tuples $c:=(q,p,p,p)$
and $d:=(w,w,w,v)$ prove the claim.

\textbf{Case 1.2} For all $u\in V$  there exist $p,v\in V$ with
$(u,v) \in Y$ such that $f(p,u) \neq f(p,v)$. Pick $m,n,u \in V$ with
$(m,n) \in X$ and $f(m,u) \neq f(n,u)$. Pick $p,v\in V$ such that $(u,v) \in
Y$ and $f(p,u) \neq f(p,v)$. If we can pick $p$ in such a way that
$(p,m), (p,n) \in X$, then since either $f(m,u)\neq f(p,u)$ or $f(n,u) \neq
f(p,u)$ we have that either $(m,p,p,p)$ or $(n,p,p,p)$ proves the
claim together with the tuple $(u,u,u,v)$. So suppose that this is
impossible. Then for any $q\in V$ with $(q,m), (q,n) \in X$ we have
$f(q,u)=f(q,v)\neq f(p,u)$, so we have that $(q,p,p,p)$ and
$(u,u,u,v)$ satisfy the claim.\\

\textbf{Case 2.} Now suppose that $a_3=a_4\in\{a_1,a_2\}$ and
$b_1=b_2\notin\{b_3,b_4\}$; without loss of generality $a_3=a_2$. Write $(b_1,b_3)\in Q_3$
and $(b_1,b_4)\in Q_4$, where $Q_3,Q_4\in\{E,N\}$.

\textbf{Case 2.1} There exists $u\in V$ such that for all $p,v,r$
with $vr\in Y$, $(u,v) \in Q_3$ and $(u,r) \in Q_4$ we have $f(p,v)=f(p,r)$.
Then one easily concludes that for all $p\in V$ and all $v,v'\in V$
with $v,v'\neq u$ we have $f(p,v)=f(p,v')$. This implies that
$f(p,v)\neq f(q,v)$ whenever $p\neq q$ and $v\neq u$. Since $f$ is
essential, there exist $p,v\in V$ with $(u,v) \in Y$ such that
$f(p,u)\neq f(p,v)$. Now pick $w,q\in V$ such that $(w,u) \in Q_3$,
$(w,v) \in Q_4$, and $(q,p) \in X$. Then $f(p,w)\neq f(q,w)$, and so the
tuples $(q,p,p,p)$ and $(w,w,u,v)$ prove the claim.

\textbf{Case 2.2} For all $u$ there exist $p,v,r$ with $(v,r) \in Y$,
$(u,v) \in Q_3$, $(u,r) \in Q_4$ and $f(p,v)\neq f(p,r)$. Pick $m,n,u$ with
$(m,n)\in X$ and $f(m,u)\neq f(n,u)$. Pick $p,v,r\in V$ such that
$(v,r)\in Y$, $(u,v) \in Q_3$, $(u,r) \in Q_4$ and $f(p,v)\neq f(p,r)$. If we
can pick $p$ in such a way that $(p,m), (p,n) \in X$, then either
$(m,p,p,p)$ and $(u,u,v,r)$ or $(n,p,p,p)$ and $(u,u,v,r)$ prove the
claim. So suppose that this is impossible. Then for any $q$ with
$(q,m), (q,n)\in X$ and all $v,r\in V$ with $(v,r)\in Y$, $(u,v) \in Q_3$,
$(u,r)\in Q_4$ we have $f(q,v)= f(q,r)$. This implies that for all such
$q$ and all $v,v'\neq u$ we have $f(q,v)=f(q,v')$. Pick $w$ such
that $(w,v)\in Q_3$, $(w,r)\in Q_4$. Pick $q$ such that $(q,p)\in X$. We have
$f(q,w)\neq f(p,w)$, and so $(q,p,p,p)$ and $(w,w,v,r)$ prove the
claim.\\

\textbf{Case 3.} To finish the proof, suppose that
$a_3=a_4\notin\{a_1,a_2\}$ and $b_1=b_2\notin\{b_3,b_4\}$. Write
$(a_3,a_1)\in P_1$, $(a_3,a_2)\in P_2$, $(b_1,b_3)\in Q_3$ and $(b_1,b_4)\in
Q_4$, where $P_i, Q_i\in\{E,N\}$.

\textbf{Case 3.1} There exists $u$ such that for all $p,v,r$ with
$(v,r) \in Y$, $(u,v) \in Q_3$ and $(u,r) \in Q_4$ we have $f(p,v)=f(p,r)$. Then
one easily concludes that for all $p\in V$ and all $v,v'\in V$ with
$v,v'\neq u$ we have $f(p,v)=f(p,v')$. This implies that $f(p,v)\neq
f(q,v)$ whenever $p\neq q$ and $v\neq u$. Since $f$ is essential,
there exist $p,v$ with $(u,v) \in Y$ such that $f(p,u) \neq f(p,v)$. Now
pick $w,m,n$ such that $wu\in Q_3$, $(w,v) \in Q_4$, $(m,n) \in X$, $mp\in
P_1$, and $(n,p) \in P_2$. Then the tuples $(m,n,p,p)$ and $(w,w,u,v)$
prove the claim.

\textbf{Case 3.2} For all $u$ there exist $p,v,r$ with $(v,r) \in Y$,
$(u,v) \in Q_3$, $(u,r) \in Q_4$ and $f(p,v)\neq f(p,r)$. Pick $m,n,u$ with
$(m,n)\in X$ and $f(m,u) \neq f(n,u)$. Pick $p,v,r$ such that $(v,r)\in Y$,
$(u,v) \in Q_3$, $(u,r) \in Q_4$ and $f(p,v)\neq f(p,r)$. If we can pick $p$
in such a way that $(p,m) \in P_1$ and $(p,n) \in P_2$, then $(m,n,p,p)$
and $(u,u,v,r)$ prove the claim, so suppose that this is impossible.
Then for any $q$ with $(q,m) \in P_1$ and $(q,n)\in P_2$ and all $v,r$ with
$(v,r) \in Y$, $(u,v) \in Q_3$, $(u,r) \in Q_4$ we have $f(q,v)= f(q,r)$. This
is easily seen to imply that for all such $q$ and all $v,v'\neq u$
we have $f(q,v)=f(q,v')$. Pick $w$ such that $(w,v) \in Q_3$, $(w,r) \in
Q_4$, and $w\neq u$. Pick $q,q'$ such that $(q,q')\in X$, $(q,p) \in P_1$ and
$(q',p) \in P_2$. We have $f(q,w)\neq f(q',w)$, and thus $(q,q',p,p)$
and $(w,w,v,r)$ prove the claim.
\end{proof}

\section{The ordered graph product Ramsey lemma}\label{sect:productRamsey}

In order to find the minimal functions which are not unary, we need to develop the Ramsey-theoretic tools that allow us to find patterns in
the behavior of such higher arity functions; this is the purpose of this section.

\begin{defn}
    Let $\Gamma_1,\ldots,\Gamma_n$ be structures. For a tuple $x$ in the cartesian product $\Gamma:=\Gamma_1\mult\cdots\mult \Gamma_n$, we write
    $x_i$ for the $i$-th coordinate of $x$. The \emph{type} of a sequence of tuples $a^1,\ldots,a^m\in \Gamma$, denoted by $\typ(a^1,\ldots,a^m)$, is the cartesian product of the types
    of $(a^1_i,\ldots,a^m_i)$ in $\Gamma_i$.
\end{defn}

\begin{defn}
    An \emph{ordered graph} is a graph with an additional total order
    $\prec$ on the vertices. An \emph{ordered graph product} is a cartesian product of ordered graphs.
\end{defn}

The following extends the definition of a canonical function on graphs, $n$-partitioned graphs, and $n$-constant graphs from Section~\ref{sect:structureInMappings} to functions on (ordered) graph products.

\begin{defn}
    Let $F_1,\ldots,F_n, Z$ be (ordered) graphs. Set $F:=F_1\mult\cdots\mult F_n$. An operation $g \colon 
    F\To Z$ is \emph{canonical} iff 
    for all
    $x,y,u,v\in
    F$ with $\typ(x,y)=\typ(u,v)$ we
    have $\typ(f(x),f(y))=\typ(f(u),f(v))$.
\end{defn}

\begin{lem}[The ordered graph product Ramsey lemma]
\label{lem:OGPRL}
    For every finite ordered graph product $F:=F_1\mult\cdots\mult F_n$ there exists a finite ordered graph product $H:=H_1\mult\cdots\mult H_n$ such that for all functions $f \colon H\To Z$ to an ordered graph $Z$ there is a
    copy $F'$ of $F$ in $H$ on which $f$ is canonical.
\end{lem}

    Note that Lemma~\ref{lem:OGPRL} is not true if the graphs are not ordered: let $n=2$, and let $I_2$ be the graph which has two vertices and no edges. Set $F_1$ and $F_2$ equal to $I_2$. Suppose $H$ exists, and order its components $H_1, H_2$ linearly. Define on the domain of $H$ the following graph $Z$: Two pairs $x,y$ are connected by an edge iff they are comparable in the product order. Set $f \colon H\To Z$ to be the identity function. Now whenever $x_i,y_i\in H_i$ are so that they induce copies $F_i'$ of $F_i$, for $i=1,2$, then $f$ is not canonical on $F_1'\mult F_2'$: for, assume without loss of generality that $x_i$ is smaller than $y_i$ in the order on $H_i$. Then
    $\typ((x_1,x_2),(y_1,y_2))=\typ((x_1,y_2),(y_1,x_2))$, but $\typ(f(x_1,x_2),f(y_1,y_2))\neq\typ(f(x_1,y_2),f(y_1,x_2))$ as $f(x_1,x_2), f(y_1,y_2)$ are connected by an edge in $Z$ but $f(x_1,y_2), f(y_1,x_2)$ are not.

    We remark that as we have seen in Section~\ref{sect:structureInMappings}, $n>1$ is necessary for this problem to appear.

\begin{proof}[Proof of Lemma~\ref{lem:OGPRL}]
    We prove the following claim: for every type $\typ(u,v)$ of $n$-tuples $u,v$ in the ordered graph product $F$ there exists a finite ordered graph product $H:=H_1\mult\cdots\mult H_n$ such that whenever $f \colon H\to Z$ is a function, then there is a copy $F'$ of $F$ in $H$ with the property that $\typ(f(x),f(y))=\typ(f(a),f(b))$  for all $x,y,a,b\in F'$ such that $\typ(x,y)=\typ(a,b)=\typ(u,v)$. Repeated use of the claim for all types of pairs in $F$ then proves the lemma.

    In fact, we prove the following more abstract statement, which clearly implies our original claim: for every type $\typ(u,v)$ of two $n$-tuples in an ordered graph product $F$ and for any $k<\omega$ there exists a finite ordered graph product $H:=H_1\mult\cdots\mult H_n$ such that whenever $\chi$ is a coloring of the pairs $(x,y)$ in $H$ with $\typ(x,y)=\typ(u,v)$ with $k$ colors, then there is a copy $F'$ of $F$ in $H$ on which the coloring is constant.

    To prove the claim, we use induction over $n$. The induction beginning $n=1$ is just a subset of the proof of Corollary~\ref{cor:coloringEdgesAndNonEdgesOfG} (as there, we had to introduce an order for the sake of the proof, and now we are already given ordered graphs). So suppose $n>1$ and that the claim holds for all $i<n$. Let $n$-tuples $u,v\in F$ defining the type be given. Set $u':=(u_1,\ldots,u_{n-1})$, and define $v'$ analogously. By induction hypothesis, there is an ordered graph product $H_1\mult\cdots\mult H_{n-1}$ such that whenever its pairs $(x',y')$ with $\typ(x',y')=\typ(u',v')$ are colored with $k$ colors, then there is a copy of $F_1\mult\cdots\mult F_{n-1}$ in $H_1\mult\cdots\mult H_{n-1}$ on which the coloring is constant. Let $m$ be the number of pairs $(x',y')$ in $H_1\mult\cdots\mult H_{n-1}$ which have type $\typ(u',v')$. By induction hypothesis, there is an ordered graph $H_{n,1}$ with the property that whenever its pairs $(x_n,y_n)$ with $\typ(x_n,y_n)=\typ(u_n,v_n)$ are colored with $k$ colors, then it contains a monochromatic copy of $F_n$. Further, there is an ordered graph $H_{n,2}$ with the property that whenever its subsets of this type are colored with $k$ colors, then it contains a monochromatic copy of $H_{n,1}$. Continue constructing ordered graphs like that, arriving at $H_n:=H_{n,m}$.
    We claim that $H:=H_1\mult\cdots\mult H_n$ has the desired properties. To see this, let a coloring $\chi$ of the pairs of $H$ of type $\typ(u,v)$  be given. Let $(x^1,y^1),\ldots,(x^m,y^m)$ be an enumeration of all the pairs in $H_1\mult\cdots\mult H_{n-1}$ which have type $\typ(u',v')$. For all $1\leq i\leq m$, define a coloring $\chi^i$ of the pairs $(p,q)$ of $H_n$ of type $\typ(u_n,v_n)$ by setting $\chi^i(p,q):=\chi(x^i\smallsmile p,y^i\smallsmile q)$, where $a\smallsmile b$ denotes the concatenation of two tuples $a,b$. By thinning out $H_n$ $m$ times, we obtain a copy $F_n'$ of $F_n$ in $H_n$ on which each coloring $\chi^i$ is constant with color $c^i$. Now by that construction, all pairs $(x^i,y^i)$ have been assigned a color $c^i$, the assignment thus being a coloring of all the pairs of type $\typ(u',v')$ in $H_1\mult\cdots\mult H_{n-1}$. By the choice of that product, there is a copy $F'_1\mult\cdots\mult F'_{n-1}$ of $F_1\mult\cdots\mult F_{n-1}$ in $H_1\mult\cdots\mult H_{n-1}$ on which that coloring is constant, say with value $r$. But that means that if $x,y\in F_1'\mult\cdots\mult F_n'$ have type $\typ(u,v)$, then $\chi(x,y)=r$, proving our statement.
\end{proof}

\section{Minimal binary functions}\label{sect:minimalBinary}

We know from Theorem~\ref{thm:bin-inj} and Lemma~\ref{lem:en} that all essential minimal functions are binary, injective, and preserve both $E$ and $N$. It is the goal of this section to determine these binary minimal functions.

Let $V$ be equipped with a total order $\prec$ in such a way that $(V;E,\prec)$ is the random ordered graph, i.e., the unique countably infinite homogeneous graph containing all finite ordered graphs (for existence and uniqueness of this structure, see e.g. \cite{Hodges}). The order $(V;\prec)$ is then isomorphic to the order of the rationals $\mathbb{Q}$.

From now on, until Proposition~\ref{prop:essentialMinimalIndependentOfOrder}, we see the random graph 
equipped with this order, in particular when talking about canonical behavior of functions. Note in this context that a function $f \colon V^k\To V$ which is canonical with respect to the language of ordered graphs need not be canonical in the language of ordinary graphs; the converse implication does not hold either.

\begin{prop}\label{prop:generatesCanonical}
    Every function $f \colon V^k\to V$ is canonical on arbitrarily large finite ordered graph products. In particular, every binary
    injection generates a binary injection which is canonical with respect to the language of ordered graphs.
\end{prop}
\begin{proof}
    The first statement is a direct consequence of the ordered graph product Ramsey lemma (Lemma~\ref{lem:OGPRL}). The second statement follows
    from the fact that there are only finitely many canonical behaviors on every finite ordered graph product, and by local closure.
\end{proof}
The following is also straightforward to verify.
\begin{prop}\label{prop:staysCanonical}
    If a function $f \colon V^k\to V$ is canonical with respect to some base structure (i.e., the random graph or the ordered random graph), then so are all functions it
    generates.
\end{prop}

Hence, by Propositions~\ref{prop:generatesCanonical} and~\ref{prop:staysCanonical} all minimal binary injections are canonical as functions on the ordered random graph. In the following, we determine those canonical behaviors of binary injections that yield minimal functions.

\begin{definition}
    Let $f \colon V^2 \rightarrow V$ be injective. If for all $(u_1,u_2),(v_1,v_2)\in V^2$ with $u_1\prec v_1$ and $u_2\prec v_2$ we have
    \begin{itemize}
        \item $(f(u_1,u_2),f(v_1,v_2)) \in E$ if and only if $(u_1,v_1) \in E$ and $(u_2,v_2) \in E$, then we say that \emph{$f$ behaves like $\min$ on input $(\prec,\prec)$}.
        \item $(f(u_1,u_2),f(v_1,v_2)) \in N$ if and only if $(u_1,v_1) \in N$ and $(u_2,v_2) \in N$, then we say that \emph{$f$ behaves like $\max$ on input $(\prec,\prec)$}.
        \item $(f(u_1,u_2),f(v_1,v_2)) \in E$ if and only if $(u_1,v_1)\in E$, then we say that \emph{$f$ behaves like $p_1$ on input $(\prec,\prec)$}.
        \item $(f(u_1,u_2),f(v_1,v_2)) \in E$ if and only if $(u_2,v_2) \in E$, then we say that \emph{$f$ behaves like $p_2$ on input $(\prec,\prec)$}.
    \end{itemize}
    Analogously, we define behavior on input $(\prec,\succ)$ using pairs $(u_1,u_2),(v_1,v_2)\in V^2$ with $u_1 \prec v_1$ and $u_2\succ v_2$.
\end{definition}

Of course, we could also have defined ``behavior on input $(\succ ,\succ )$'' and ``behavior on input $(\succ ,\prec )$''; however, behavior on input $(\succ ,\succ )$ equals behavior on input $(\prec ,\prec )$, and behavior on input $(\succ ,\prec )$ equals behavior on input $(\prec ,\succ )$. Thus, there are only two kinds of inputs to be considered, namely the ``straight input" $(\prec,\prec)$ and the ``twisted input"  $(\prec ,\succ )$.

\begin{prop}\label{prop:binaryBehaviourOnInputBLaBla}
Let $f \colon V^2 \rightarrow V$ be injective and canonical as a function on the ordered random graph, and suppose it preserves $E$ and $N$. Then it behaves like $\min$, $\max$, $p_1$ or $p_2$ on input $(\prec ,\prec )$ (and similarly on input $(\prec ,\succ)$).
\end{prop}
\begin{proof}
    By definition of the term canonical; one only needs to enumerate all possible types $\typ(x,y)$ of pairs $x,y\in V^2$ with respect to the ordered random graph.
\end{proof}

We remark that the four possibilities correspond to the four binary operations $g$ on the two-element domain $\{E,N\}$ that are \emph{idempotent}, i.e., that satisfy that $g(E,E)=E$ and $g(N,N)=N$.

\begin{defn}
    If $f \colon V^2 \rightarrow V$ behaves like $X$ on input $(\prec ,\prec )$ and like $Y$ on input $(\prec ,\succ )$, where $X,Y\in\{\max,\min, p_1,p_2\}$, then we say that $f$ is of \emph{type $X / Y$}.
\end{defn}

Observe that in Proposition~\ref{prop:binaryBehaviourOnInputBLaBla}, we did not care about the fact that a canonical injection
$f \colon V^2 \rightarrow V$ also behaves regularly with respect to the order: The latter implies, for example, that $f$ is either strictly increasing or decreasing with respect to the pointwise order, i.e., $x_1\prec y_1$ and $x_2\prec y_2$  either always implies $f(x_1,x_2)\prec f(y_1,y_2)$, or it always implies $f(x_1,x_2)\succ f(y_1,y_2)$. Fix from now on any automorphism $\alpha$ of the graph $G$ that reverses the order on $V$. By applying $\alpha$ to $f$ if necessary, we may assume that $f$ is strictly increasing, which will be a tacit assumption from now on. Having that, one easily checks that $f$ satisfies one of the implications
$$
x_1\prec y_1\wedge x_2\neq y_2\rightarrow f(x_1,x_2)\prec f(y_1,y_2)
$$
and
$$
x_1\neq y_1\wedge x_2\prec y_2 \rightarrow f(x_1,x_2)\prec f(y_1,y_2).
$$
In the first case, we say that \emph{$f$ obeys $p_1$ for the order}, in the second case \emph{$f$ obeys $p_2$ for the order}. By switching the variables of $f$, we may always assume that $f$ obeys $p_1$ for the order, and we shall do so from now on.

We will now prove that minimal binary canonical injections are never of mixed type, i.e., they have to behave the same way on straight and twisted inputs.

\begin{prop}\label{prop:noMixed}
    Let $f \colon V^2 \rightarrow V $ be injective and canonical as a function on the ordered random graph. Suppose moreover that $f$ is of type $X / Y$, where $X,Y\in\{\max,\min,p_1,p_2\}$ and $X\neq Y$. Then $f$ is not minimal.
\end{prop}
\begin{proof}
    Suppose first that $f$ is of type $\max / p_i$ or of type $p_i / \max$, where $i\in\{1,2\}$. 
     We claim that $f$ generates a binary injective canonical function $g$ which is of type $\max / \max $. Clearly, all binary injective canonical functions generated by $g$ then are also of type $\max / \max $, so $g$ cannot generate $f$, which shows that $f$ is not minimal.
    Assume without loss of generality that $f$ is of type $\max / p_i$, and note that we assume that $f$ obeys $p_1$ for the order. Set $h(u,v):=f(u,\alpha(v))$. Then $h$ behaves like $p_i$ on input
    $(\prec ,\prec )$ and like $\max$ on input $(\prec ,\succ )$; moreover, $f(x_1,x_2)\prec f(y_1,y_2)$ iff $h(x_1,x_2)\prec h(y_1,y_2)$, for all
     $x_1\neq y_1$ and $x_2\neq y_2$. We then have that $g(u,v):=f(f(u,v),h(u,v))$ is of type $\max / \max$, finishing the proof of our claim. 

    If $f$ is of type $\min / p_i$ or of type $p_i / \min$, where $i\in\{1,2\}$, then the dual proof works.

    Consider the case where $f$ is of type $\max / \min$ or of type $\min / \max$. 
    Assume without loss of generality that $f$ is of type $\max / \min$, and remember that we assume that $f$ obeys $p_1$ for the order. Consider $h(u,v):=f(f(u,v),\alpha(v))$. Then $h$ is of type $p_2 / p_2$, so it cannot reproduce $f$. Hence $f$ is not minimal.

    To finish the proof, suppose that $f$ is of type $p_1 / p_2$ or of type $p_2 / p_1$.
    If $f$ is of type $p_1 / p_2$, then $h(u,v):=f(f(u,v),\alpha(v))$ is of type $p_2 / p_2$ and cannot reproduce $f$. If $f$ is of type $p_2 / p_1$, then $g(u,v):=f(u,\alpha(v))$ is of type $p_1 / p_2$ and still obeys $p_1$ for the order; hence, we are back in the first case.
\end{proof}

This motivates the following definition.

\begin{definition}
    Let $f \colon V^2 \rightarrow V$. We say that $f$ \emph{behaves like $\min$ ($\max$, $p_1$, $p_2$)
    on input $(\neq,\neq)$} iff it behaves like $\min$ ($\max$, $p_1$, $p_2$) both on input $(\prec ,\prec )$ and on input $(\prec,\succ )$. We also say that $f$ \emph{is of type $\min$ ($\max$, $p_1$, $p_2$)}. If $f$ is of type $p_1$ or $p_2$ then we also say that $f$ is of type \emph{projection}.
\end{definition}

Our observations so far can be summarized as follows.

\begin{prop}\label{prop:essentialMinimalIndependentOfOrder}
Let $f \colon V^2 \rightarrow V $ be essential and minimal. Then it is injective, canonical as a function on the (non-ordered) random graph and behaves like $\min$, $\max$, $p_1$ or $p_2$ on input $(\neq,\neq)$.
\end{prop}
\begin{proof}
	We know that $f$ generates a binary injection which is canonical as a function on the random ordered graph, hence it is itself such a function. Since by Proposition~\ref{prop:noMixed}, $f$ cannot have a ``mixed'' behavior, it behaves like $\min$, $\max$, $p_1$ or $p_2$, and hence is also canonical as a function on the non-ordered random graph.
\end{proof}

In the following, we will thus forget about the order that we imposed on the random graph, and use the terms ``canonical'' and ``behavior'' relative to $G$. We now consider further types of tuples $x,y\in V^2$: So far, we did not look at the case where $x_1=y_1$ or $x_2=y_2$.

\begin{definition}
    Let $f \colon V^2 \rightarrow V $. We say that $f$ \emph{behaves like $e_E$ ($e_N$, $\id$, $\flip$) on input $(\neq,=)$}
    iff for every fixed $c\in V$, the function $g(x):=f(x,c)$ behaves like $e_E$ ($e_N$, $\id$, $\flip$). Similarly we define behavior on input $(=,\neq)$.
\end{definition}

If $f$ is canonical and injective, then it behaves like one of the mentioned functions on input $(\neq,=)$ and $(=,\neq)$, respectively. This motivates the following

\begin{definition}
We say that $f \colon V^2\To V$ is \emph{of type  $E/N$} iff $f$ behaves like $e_E$ on input $(\neq,=)$ and like $e_N$ on input $(=,\neq)$. Similarly we define the types $E/E$, $N/E$, $E/\id$, $E/\flip$, etc. Moreover, we say that $f$ is \emph{balanced} iff it is of type $\id/\id$, we say it is \emph{$E$-dominated} iff it is of type $E/E$, and we say it is \emph{$N$-dominated} iff it is of type $N/N$.
\end{definition}

In the following proposition we finally characterize those canonical behaviors that yield minimal functions.

\begin{prop}\label{thm:minimalFunctions}
    The essential minimal operations on $G$ are precisely the binary
    injective canonical operations of the following types:
       \begin{enumerate}
        \item Projection and balanced.
        \item $\max$ and balanced.
        \item $\min$ and balanced.
        \item $\max$ and $E$-dominated.
        \item $\min$ and $N$-dominated.
        \item Projection and $E$-dominated.
        \item Projection and $N$-dominated.
        \item $p_2$ and $E/\id$, or $p_1$ and $\id/E$.
        \item $p_2$ and $N/\id$, or $p_1$ and $\id/N$.
   \end{enumerate}
   Moreover, these 9 different kinds of minimal functions do not generate one another. Furthermore, any two functions in the same group do generate one another.
\end{prop}
\begin{proof}
By Proposition~\ref{prop:essentialMinimalIndependentOfOrder} we know that all minimal essential functions are necessarily canonical binary injections of type $\min$, $\max$, or projection. Therefore, we must show that out of those functions, the minimal ones  are precisely those listed above. Let henceforth $f$ be a canonical binary injection of type $\min$, $\max$, or projection.

Let us first prove that if $f$ is listed above, then it is indeed minimal. To this end, observe first that by the homogeneity of $G$ and local closure, $f$ then generates all other functions in its class of the theorem. Next note the following facts which can easily be proven by a standard induction over terms.
\begin{itemize}
\item   Any binary essential function generated by a binary canonical injection of type $\min$, $\max$, or projection, respectively, is of the same type.
\item  Any binary essential function generated by a binary canonical injection that is balanced and preserves $E$ and $N$ is balanced.
\end{itemize}

It follows immediately that the if $f$ belongs to the first three classes of the proposition, then it is minimal. 

It is easy to verify that any binary essential function generated by an $E$-dominated binary canonical injection of type $\max$ is $E$-dominated. Dually, any binary essential function generated by an $N$-dominated binary canonical injection of type $\min$ is $N$-dominated. These two facts imply minimality in case $f$ belongs to items (4) or (5).

Observe next that any binary essential function generated by an $E$-dominated binary canonical injection of type projection is $E$-dominated. Dually, any binary essential function generated by an $N$-dominated binary canonical injection of type projection is $N$-dominated. This implies  minimality for the case where $f$ belongs to items (6) or (7).

To prove minimality for the case where $f$ falls into items (8) or (9) we claim the following: Any binary essential function generated by a binary canonical injection of type $E / \id$ and $p_2$ is either of the same type or of type $\id / E$ and $p_1$. Dually, any binary essential function generated by a binary canonical injection of type $N / \id$ and $p_2$ generates is either of the same type or of type $\id / N$ and $p_1$. To see this, let $f(u,v)$ be of type  $E/\id$ and $p_2$. $f(v,u)$ is of type $\id/E$ and $p_1$. Both $f(u,f(u,v))$ and $f(v,f(u,v))$ are of type $E/\id$ and $p_2$. So is $f(f(u,v),v)$. The function $f(f(u,v),u)$ is of type $\id/E$ and $p_1$. Finally, $f(f(u,v),f(v,u))$ also is of type $\id/E$ and $p_1$, so $f$ cannot generate any new typesets.

Next we show that if $f$ does not belong to any of the listed classes, then it is not minimal. 

Suppose first it is of type $\max$. We claim that if $f$ is not balanced or $E$-dominated, then $f$ is not minimal. We go through all possibilities: If $f$ is of type $E/\id$, then $g(x,y):=f(f(x,y),x)$ is $E$-dominated. By our observation above, $g$ cannot reproduce $f$. If $f$ is of type $E/N$, then $g$ is $E$-dominated as well. So it is if $f$ is of type $E/\flip$. If $f$ is of type $N/\id$, then $g(x,y):=f(x,f(x,y))$ is balanced, so $f$ is not minimal by the above. If $f$ is of type $N/\flip$, then $g$ is balanced as well. If $f$ is of type $\id/\flip$ or of type $\flip/\flip$, then $g(x,y):=f(x,f(x,y))$ is of type $E/\id$, which we have already shown not to be minimal.  By symmetry, if we switch the arguments in a type of $f$, e.g., if $f$ is of type $\id/E$, then $f$ is not minimal either. We have thus covered all possible types.

The dual argument works if $f$ is a binary canonical injection of type $\min$: If $f$ is not balanced or $N$-dominated, then $f$ is not minimal.

Suppose now that $f$ is of type $p_1$. We claim that if $f$ is not balanced, $E$-dominated, $N$-dominated, of type $\id/E$, or of type $\id/N$, then $f$ is not minimal. To see this, we distinguish all possible cases:  If $f$ is of type $E/\id$, $E/\flip$, $\flip/\id$, or $\flip/\flip$, then $g(x,y):=f(x,f(x,y))$ is balanced and cannot reproduce $f$. If it is of type $E/N$ or $\id/\flip$, then $g$ is of type $E/\id$, and we are back in the preceding case. Dually, if $f$ is of type $N/\id$ or $N/\flip$, then $g$ is balanced. If it is of type $N/E$, then $g$ is of type $N/\id$, bringing us back to the preceding case. If it is of type $\flip/E$, then $g$ is of type $\id/E$ and $p_1$, and hence cannot reproduce $f$ by the above. The dual argument works if $f$ is of type $\flip/N$.

Finally, et $f$ be of type $p_2$. Then the same argument as above shows that if $f$ is not balanced, $E$-dominated, $N$-dominated, of type $E/\id$, or of type $N/\id$, then $f$ is not minimal. This finishes the proof.
\end{proof}

To summarize, we now restate and prove Theorems~\ref{thm:mainUnary} and ~\ref{thm:main} in the following

\begin{theorem}[Summary of Theorems~\ref{thm:mainUnary} and~\ref{thm:main}]
Any minimal function $e$ on the random graph is equivalent to exactly one of the following operations:
a constant operation; $e_N$; $e_E$; $\flip$; $\sw$; or
\begin{itemize}
\item[(6)] a binary injection of type $p_1$ that is balanced in both arguments;
\item[(7)]  a binary injection of type $\max$ that is balanced in both arguments;
\item[(8)]  a binary injection of type $\max$ that is $E$-dominated in both arguments;
\item[(9)]  a binary injection of type $p_1$ that is $E$-dominated in both arguments;
\item[(10)]  a binary injection of type $p_1$ that is balanced in the first
and $E$-dominated in the second argument;
\end{itemize}
or to one of the duals of the last four operations (the operation in (6) is self-dual).
\end{theorem}
\begin{proof}
If $e$ is not essential, then it generates, and hence is equivalent to, a constant operation, $e_N$, $e_E$, $\flip$  or $\sw$; this is the content
of Theorem~\ref{thm:mainUnary}, which we already proved in Section~\ref{sect:minimalUnary}. We also have argued that these functions do not
generate each other.

If $e$ is essential, then it must preserve
$E$ and $N$, by Lemma~\ref{lem:en},
and it must be binary and injective by Theorem~\ref{thm:bin-inj}.
By Proposition~\ref{thm:minimalFunctions} it must be canonical and of one of the 9
types listed there; moreover, Proposition~\ref{thm:minimalFunctions} shows that
functions of different types do not generate each other. Observe that in Proposition~\ref{thm:minimalFunctions}, class (1) is (6) here,
(2) is (7) here, (3) is the dual of (7) here, (4) is (8) here, (5) is the dual of (8) here, (6) is (9) here, (7) is the dual of (9) here, (8) is (10) here, and (9) is the dual of (10) here.
\end{proof}

We conclude this paper by remarking that the relations that are preserved by one of the
essential operations in Theorems~\ref{thm:mainUnary} and~\ref{thm:main} (i.e., the relations of the structures in Corollary~\ref{thm:mainForStructures}) also have syntactic descriptions.
For instance, it is not hard to show (see~\cite{Maximal}) that a relation $R$ with a first-order definition in $G$
is preserved by a binary operation of type $\min$ that is $N$-dominated in both arguments if and only if $R$ can be
defined by a quantifier-free Horn formula over $(V; E,=)$
(i.e., by a quantifier-free formula in conjunctive normal form where
each clause contains at most one literal of the form $E(x,y)$ or $x=y$).

\bibliographystyle{plain}
\bibliography{global.bib}

\end{document}